\documentclass[reqno,a4paper,11pt]{amsart}
\usepackage[utf8]{inputenc}

\calclayout

\usepackage[english]{babel}
\usepackage{helvet}          
\usepackage{courier}         
\usepackage{upgreek}
\usepackage{amsthm}
\usepackage{bm} 
\usepackage{bbm} 
\usepackage{float} 
\usepackage{amsmath, amsthm, amssymb, tikz, enumerate, dsfont, csquotes, xcolor, tabularx, diagbox}
\usepackage{booktabs}
\usepackage{neuralnetwork}
\usepackage{pgfplots}
\pgfplotsset{compat=1.15}
\usepackage{mathrsfs}
\usepackage[colorlinks=true,linkcolor=blue,citecolor=magenta]{hyperref}
\usepackage[T1]{fontenc}
\usepackage{xcolor}

\usetikzlibrary{calc, through, intersections, shapes.geometric}
\usepackage{seqsplit}

\DeclareMathOperator{\sgn}{sgn}

\newcommand{\esp}[1]{\mathbb{E}\left[#1\right]}
\newcommand{\prob}[1]{\mathbb{P}(#1)}
\newcommand{\btt}{\triangleright}
\newcommand{\wst}{\triangleleft}

\newtheorem{lemma}{Lemma}
\newtheorem{theorem}[lemma]{Theorem}
\newtheorem{claim}[lemma]{Claim}

\newtheorem{definition}[lemma]{Definition}

\newtheorem{proposition}[lemma]{Proposition}
\newtheorem{assumption}[lemma]{Assumption}
\newtheorem{example}[lemma]{Example}
\theoremstyle{remark}
\newtheorem{remark}[lemma]{Remark}

\newcommand{\mc}[1]{{\mathcal #1}} 
\newcommand{\mf}[1]{{\mathfrak #1}} 
\newcommand{\mb}[1]{{\mathbf #1}} 
\newcommand{\bb}[1]{{\mathbb #1}}

\newcommand{\wt}[1]{{\widetilde{#1}}} 
\newcommand{\ind}{\mathbbm{1}}

\newcommand{\f}{{\mf f}}

\numberwithin{equation}{section}

\newcommand{\dado}[6]{
    \node (pol#1) [draw=black, thick, fill=blue!60!green,rotate=0, fill opacity=0.2, text opacity=1,minimum size=3cm,regular polygon, regular polygon sides=#2] at #4 {$#1$};

    \foreach \n [count=\nu from 1] in #3
        \coordinate (#1coord\nu) at (pol#1.corner \nu);
    \coordinate (#1coord0) at (pol#1.corner #2);

    \foreach \n [count=\nu from 1, count=\nuu from 0] in #3
        \node[label={[label distance=2cm]},anchor=(\nu+(#2+1)/2 #5)*(360/#2)]at($(#1coord\nuu)!0.5!(#1coord\nu)$){$\n$};

    \foreach \n [count=\nu from 1, count=\nuu from 0] in #6
        \node[ anchor=(\nu+1/2 #5)*(360/#2)]at($(#1coord\nuu)!0.5!(#1coord\nu)$){$\n$};
}

\newcommand{\var}[1]{\operatorname{Var}\left(#1\right)}
\newcommand{\cov}[1]{\operatorname{Cov}\left(#1\right)}
\newcommand{\corr}[1]{\operatorname{Corr}\left(#1\right)}

\makeatletter
\renewenvironment{proof}[1][Proof]{\par
\pushQED{\qed}%
\normalfont \topsep6\p@\@plus6\p@\relax
\trivlist
\item\relax
{\itshape
#1\@addpunct{.}}\hspace\labelsep\ignorespaces
}{%
\popQED\endtrivlist\@endpefalse
}
\makeatother

\DeclareMathOperator{\ee}{\rm e}

\allowdisplaybreaks

\let\oldtocsection=\tocsection
\let\oldtocsubsection=\tocsubsection
\let\oldtocsubsubsection=\tocsubsubsection
\renewcommand{\tocsection}[2]{\hspace{0em}\oldtocsection{#1}{#2}}
\renewcommand{\tocsubsection}[2]{\hspace{1em}\oldtocsubsection{#1}{#2}}
\renewcommand{\tocsubsubsection}[2]{\hspace{2em}\oldtocsubsubsection{#1}{#2}}
\DeclareRobustCommand{\SkipTocEntry}[5]{}

\begin{document}

\title{A Central Limit Theorem for Intransitive Dice}

\author[L.\ Coelho]{Luis\ G.\ Coelho}
\address[LC]{Faculdade de Filosofia, Ci\^encias e Letras de Ribeir\~ao Preto, Universidade de S\~ao Paulo, Brazil}
\email{luisguilhermecoelho@usp.br}

\author[T.~Franco]{Tertuliano Franco}
\address[TF]{UFBA\\
 Instituto de Matem\'atica, Campus de Ondina, Av. Adhemar de Barros, S/N. CEP 40170-110\\
Salvador, Brazil}
\email{tertu@ufba.br}

\author[L.\ Lima]{Lael\ V.\ Lima}
\address[LL]{Instituto de Matemática, Estatística e Computação Científica, Universidade Estadual de Campinas, Brazil}
\email{l176809@dac.unicamp.br}

\author[J.\ de Paula]{João\ P.\ C.\ de Paula}
\address[JPa]{Instituto de Matemática, Estatística e Computação Científica, Universidade Estadual de Campinas, Brazil}
\email{j246527@dac.unicamp.br}

\author[J.\ Pimenta]{João\ V.\ A.\ Pimenta}
\address[JP]{Instituto de F\'isica de S\~ao Carlos, Universidade de S\~ao Paulo, Brazil}
\email{joaovictorpimenta@usp.br}

\author[G.\ Silva]{Guilherme\ L.\ F.\ Silva}
\address[GS]{Instituto de Ciências Matemáticas e de Computação, Universidade de S\~ao Paulo, Brazil.}
\email{silvag@usp.br}

\author[D.\ Ungaretti]{Daniel Ungaretti}
\address[DU]{Instituto de Matem\'atica, Universidade Federal do Rio de Janeiro, Brazil}
\email{daniel@im.ufrj.br}

\date{}

\begin{abstract}
    Intransitive dice $D^{(1)}, \ldots, D^{(\ell)}$ are dice such that $D^{(1)}$ has advantage when played against $D^{(2)}$, dice $D^{(2)}$ has advantage when played against $D^{(3)}$ and so on, up to $D^{(\ell)}$, which has advantage over $D^{(1)}$.  In this twofold work, we first present (deterministic) results on the existence of general intransitive dice. Second and mainly, a central limit theorem for the vector of normalized victories of a die against the next one in the list when the faces of a die are i.i.d.\ random variables and all dice are independent, but different dice may have distinct distributions associated with them, as well as they may have distinct numbers of faces. Exploiting this central limit theorem, we derive two major consequences. First, we are able to obtain first order exponential asymptotics for the number of $\ell$-tuples of intransitive dice, when the number of faces of the dice grows. Second, we obtain a criterion to ensure that the asymptotic  probability of observing intransitive dice is null, which applies to many cases, including all continuous distributions and many discrete ones.

\end{abstract}

\maketitle

\tableofcontents
\allowdisplaybreaks

\section{Introduction}

Intransitivity is an inherent facet of nature, it is part of the equilibrium in evolutionary dynamics, where different relations between predators and prey create the balance for common existence. This phenomenon is noted for instance in the eighteenth-century Condorcet's paradox, in which three candidates are intransitive in the sense that the candidate $A$ wins when running against $B$, the candidate $B$ wins when running against candidate $C$ and the candidate $C$ wins when running against candidate $A$. It is worth mentioning that Condorcet's paradox is intrinsically related to the classical Arrow's Theorem. Intransitivity also manifests itself in sports leagues, network relations, interactions between different medications, and an ever-expanding array of scenarios.

While being a fundamental mathematical concept, intransitivity can lead to intriguing outcomes even in simple models. Consider a basic dice game as an example: there are two players, each tosses a (possibly different) die and the one with the highest outcome wins. It is possible to construct three dice, $ A $, $ B $ and $ C $, for which $ A $ is better than $ B $ (in the sense that the player with die $ A $ has a higher chance of winning against the player with die $ B $), $ B $ is better than $ C $ and $ C $ is better than $ A $?
What about constructing an intransitive chain of more than three dice? And dice with a very large number of faces?

To the best of our knowledge, the first examples of intransitive chains of dice first appeared in the literature in Martin Gadner's column \cite{Gardner}, where the author mentions a previous construction from the sixties by Bradley Effron, consisting of the dice
\begin{align*}
&A=(0, 0, 4, 4, 4, 4),\quad
B=(3, 3, 3, 3, 3, 3),\\
&C = (2, 2, 2, 2, 6, 6),\quad
D = (1, 1, 1, 5, 5, 5).
\end{align*}
For those dice, it holds that the probability of $A$ beats $B$, $B$ beats $C$, $C$ beats $D$ and $D$ beats $A$ are all equal to $2/3$. Intransitive dice are also natural	 examples of the
\textit{Steinhaus and Trybu{\l}a's paradox} (see \cite{Steinhaus, Lebedev}), consisting on the existence of independent random variables
$X$, $Y$, and $Z$ such that $\bb P(X > Y ) > 1/2$,
$\bb P(Y > Z ) > 1/2$, and $\bb P(Z > X ) > 1/2$.
The property of intransitivity can be found in various domains, such as Statistics \cite{Brown_Hett} and voting systems \cite{PTRF}, to mention only a few.

From a probabilistic point of view, there has been a recent upraise of interest
in intransitive dice phenomena. In part, such a recent trend started with a
discussion by Conrey, Gabbard, Grant, Liu and Morrisson in  \cite{CGGLM}. Therein, the authors considered a model of random dice where the $n$ faces of a given random die are given by uniformly choosing $n$ entries among positive integers conditioned to sum to $n(n+1)/2$, which they called \textit{a balanced model}. For instance, for $n=4$, the faces of a die are chosen by picking uniformly one of the  multisets below:
\begin{equation*}
(1, 1, 4, 4),\quad (1, 2, 3, 4),\quad (1, 3, 3, 3),\quad (2, 2, 2, 4), \quad (2, 2, 3, 3)\,.
\end{equation*}
Following \cite{CGGLM}, choose a set of three random independent dice from the balanced model. With a strong support by computational evidence, in \cite{CGGLM} they also propose two conjectures: first, that the asymptotic probability of ties between any two of the three dice is zero, and second, that the asymptotic probability that these three dice form an intransitive chain is $1/4$ (see also \cite{Morrison} which evaluates some exact probabilities for three and four dice). These two conjectures were later proved by a Polymath project \cite{polymath}, and several other results on the balanced model and other related models of random dice with constrained sum faces were further explored in various recent works \cite{PTRF2}. Intransitive dice may also be interpreted through tournament graphs, see \cite{Akin21, AkinSaccamano21} for further results in such direction.

In the present work, we ask ourselves about the existence of intransitive chains of dice both from deterministic and probabilistic perspectives. We consider an arbitrary number of dice, each of them with an arbitrary number of faces, without any constraints regarding the sum of the faces.

The first part of this paper is devoted to studying the existence of intransitive dice: in a deterministic setup that does not allow for ties among different dice faces, we rediscover a characterization from \cite{Schaefer}, describing when an ordered collection of intransitive dice exists, in terms of the size of the dice and the number of different entries used for the faces. In short terms, such a result says that there are always intransitive collections with arbitrarily large number of faces, provided each die has at least $3$ faces. Naturally, one is faced with the question of how many of these collections there are. As a first novel result, we are able to show that the proportion of ordered collections of intransitive dice among all possible collections (not necessarily intransitive) decays with the number of faces of the dice, and we compute the leading term in the exponential decay rate asymptotics explicitly.

For both the previously mentioned existence results and decay of the proportion of intransitive dice, the key observation is a bijection between collections of dice, not necessarily intransitive, and words with appropriate number of letters. We explore this connection to construct, from a given collection of intransitive dice, a new collection with a larger number of faces of each die or with a larger number of dice, while preserving the intransitivity. This construction is algorithmic, and as we mentioned it is based on the connection between intransitive dice and words with a particular combinatorial property which may be of independent interest.

The second part of this paper deals with models of random dice, where the entries on the faces of a given die are independent random variables with the same distribution, but the distributions generating different dice may vary. Our main interest lies in determining the chance that a finite collection of random dice is intransitive, when the number of faces of each die grows large. To do so, we prove a central limit theorem for the vector of number of victories of the faces of die against the faces of the next die in the list (whose entries are strongly correlated). The proof of this CLT  is based on the moment method, where the crucial steps consist of a combinatorial bijection between moments and an appropriate subclass of graphs, and a careful estimate of the number of such graphs. 

The vector of victories is actually connected to intransitivity, which is simple to illustrate when each die has $n$ faces with no ties: in this particular case, intransitity of the list of dice is  equivalent to the fact that each entry of the vector of victories is larger than $n^2/2$.
This can be properly generalized to account for possible ties among entries, and when combined with the CLT obtained, we are able to deduce a criterion to compute the probability of intransitivity in terms of a Gaussian probability.

Such result is obtained under appropriate but natural conditions on the distribution of the random variables determining the faces. \textit{Grosso modo}, the conditions are bounds on the number of ties and on the variance of victories of a die against another one in the list, avoiding degeneracy in the central limit setting, with appropriate first order asymptotics for their means.  These mild conditions cover many situations. For instance, it  includes the  scenarios where  all dice have same distribution (including all continuous distributions and many discrete ones), and also situations where the underlying distributions of faces depend on scaling parameters. And also the case where dice have different distributions in some cases. We also provide a way of constructing asymptotic intransitive dice (not satisfying the previous conditions, of course), which is argued via a concentration inequality.

We now move forward to the discussion of our main findings.

\section{Statement of results}\label{s:statements}

We split the discussion of these major results into two subsections, first for deterministic dice and then for random dice.
As we hope to convey with this text, simple models of dice display rather interesting and rich aspects worth investigating deeper. However, many interesting phenomena may depend on somewhat subtle specific features of the model considered. Nevertheless, questions surrounding intransitive dice phenomena are rather simple to state. For the latter reason, we mostly introduce new terminology and notation along the text, reserving formal definitions solely for more technical assumptions needed. For convenience, such definitions along the text are highlighted in bold.

\subsection{Main results for deterministic models of dice}\label{sec:determdice}
\hfill

An $n$\textbf{-sided die} is a pair $(D,X)$, where $D=(D_1,\hdots, D_n)$ is a real-valued vector where each $D_k$ represents the number on the $k$-th face, and $X$ is a random variable taking values on $[n] \coloneqq\{1, 2, \hdots, n\}$ that represents the label of the face in the outcome of a toss. The number $n$ is the number of faces, or simply size, of the die $D$.
The die is said to roll the face $k$ with probability $\prob{X=k}$ and results $D_k$. If this probability equals $1/n$ for every $k$, the
die is \textbf{honest} or \textbf{fair}. Otherwise the die is \textbf{unfair} or \textbf{biased}. If there is no ambiguity, the die will be denoted as $D$, and in that case, it is
useful to denote the random result of $D$ in a roll by $D_X$. Thus, in general, the entries of $D$ need not be integer-valued, nor even positive. We reserve capital letters $A$, $B$, $C$ etc.\ to represent dice, and lower indices $A_i$, $B_i$ etc.\ to represent a entries of the dice $A$, $B$. It is also useful to distinguish different dice with an upper index, writing for instance $D^{(1)}$, $D^{(2)}$ etc, and the corresponding entries by $D^{(1)}_i$, $D^{(2)}_i$ etc.

A die \(A\) is said to be {\bfseries better than} a die \(B\), and it is
denoted by \(A \btt B\) if the probability of \(A\) rolling a higher value than
\(B\) is greater than the probability of \(B\) rolling a higher value than
\(A\). To the same extent, the die \(B\) is said to be {\bfseries worse than}
\(A\), and it is denoted by \(B\wst A\).

In mathematical terms, one way to verify whether a fair die $A$ is better than a fair die $B$ is by counting against how many faces of $B$ a given face of $A$ wins, summing the result over all possible faces of $A$, and comparing with the count we obtain when we do the same interchanging the roles of $A$ and $B$. In other words, $A\btt B$ if, and only if, the inequality
$$
\sum_{A_i>B_j}1 \;>\;\sum_{B_j>A_i}1
$$
is satisfied. With $n_A$ and $n_B$ being the number of faces of $A$ and $B$, respectively, there are in total $n_An_B$ pairs of faces from $A$ and $B$ to compare, and $A\btt B$ if, and only if,
\begin{equation}\label{eq:fundcomparison}
\sum_{A_i>B_j}1 \;>\;
\frac{1}{2}n_An_B-\frac{1}{2}\sum_{A_i=B_j}1\,.
\end{equation}

An ordered collection of dice $\mb D=(D^{(1)},\dots,D^{(\ell)})$ is said to be \textbf{intransitive} if \(D^{(1)} \btt
\cdots \btt D^{(\ell)}\btt D^{(1)}\).  Note that while $\btt$ is an asymmetric relation,
it is not necessarily transitive, so it does not define an order relation. When
computing whether a given collection $\mb D=(D^{(1)},\hdots,D^{(\ell)})$ of
dice is intransitive, the ordering of the entries does matter, and it is possible that $\mb D$ is not
intransitive, but for some permutation $\sigma$ of length $\ell$ a reordering $(D^{(\sigma(1))},\hdots, D^{(\sigma(\ell ))})$ is intransitive. We will be interested in existence results for deterministic collections $\mb D$, and in asymptotic results when the distributions of the entries of each die are rather arbitrary, so this ordering will not be relevant in any essential way.

By a \textbf{no-tie collection} of dice we mean that no pair of faces, either from the same die or from different dice, shares the same number.

Our first two results concern intransitive families of deterministic dice. The first one deals with the existence of intransitive, fair dice.
        \begin{theorem}\label{different_faces}
        Consider dice whose face entries are positive integers.
                For every \(\ell \geq 3\) and \(n \geq 3\) there exists a no-tie collection of $\ell$ honest $n$-sided dice which is intransitive.
                Furthermore, for any $\ell\geq 3$ there does not exist a no-tie family of $\ell$ honest $2$-sided dice which is intransitive.
        \end{theorem}

This result was already proved in \cite[Theorem~2.1]{Schaefer}, with a direct construction. We provide a different proof, with an approach using a bijection between dice and words, which we explain in Section~\ref{sec:dicewords}. This approach also lies at the core of many of our novel results, so we decided to also present Theorem~\ref{different_faces} here, to illustrate our techniques.

The notion of a die $A$ being better than $B$ is not a relation on the specific numbers on their faces, but rather between the relative ordering of these numbers. For instance, the die $A=(2,4,9,10,11)$ is better than the die $B=(1,5,7,9,10)$. Now, increase, say, the first entry of $A$ to a new die $\widetilde A=(x,4,9,10,11)$ with any choice $x={3,4}$. Then when we choose and roll either one of dice $A$ or $\widetilde A$, the chance of winning against a roll of die $B$ is the same. So, in terms of \textit{chance of winning} against $B$, dice $A$ and $\widetilde A$ are indistinguishable.

In that sense, when we talk about comparison of $\ell$ non-tie dice with $n$ faces each, it suffices to distribute the numbers in the set $[\ell n]$ among the faces of the dice, without repetition. In fact, the proof of the existence claim in Theorem~\ref{different_faces} is inductive/constructive, and shows that such $\ell$ honest $n$-sided dice can always be chosen with distinct entries in $[\ell n]$.

For a given choice of positive integers $n_1,\hdots,n_\ell$, let $\mc D(n_1,\hdots, n_\ell)$ be the set of collections of dice $\mb D=(D^{(1)},\hdots, D^{(\ell)})$ for which
$D^{(j)}$ has exactly $n_j$ faces, and where each number in $[n_1+\cdots +n_\ell]$ appears exactly once in the faces in $\mb D$. In other words, the dice are filled with numbers in $[n_1+\cdots +n_\ell]$, without repetition. Observe that with this definition, in $\mc D(n_1,\hdots, n_\ell)$ we do not distinguish between the ordering of faces in each die. Or, alternatively, dice in $\mc D(n_1,\hdots, n_\ell)$ are always viewed in increasing order, so that for instance the dice $(1,2,3)$ and $(2,1,3)$ are the same and are always represented by $(1,2,3)$. But we do distinguish between orderings within a collection, so that the collections $\mb D=((1,2,4),(3,5,6)),\widehat{\mb D}=((3,5,6),(1,2,4))$ are distinct elements of $\mc D(3,3)$.  In other words,
\begin{multline*}
\mc D(n_1,\hdots, n_\ell)\;\coloneqq\; \Big\{
\mb D=(D^{(1)},\hdots, D^{(\ell)}): D^{(j)}=(D^{(j)}_1,\hdots D^{(j)}_{n_j})\in \mathbb Z^{n_j}, \\ 0<D^{(j)}_1<\cdots<D^{(j)}_{n_j} \text{ for } j=1,\hdots,\ell,
D^{(j_1)}_{i_1}\neq D^{(j_2)}_{i_2} \text{ for }j_1\neq j_2, \{D^{(j)}_i\}_{i,j}=[n_1+\cdots+n_\ell]
\Big\}.
\end{multline*}

We denote by $\mc D_\btt (n_1,\hdots,n_\ell)$ the subset of $\mc D (n_1,\hdots,n_\ell)$ that consists of intransitive dice, that is,
$$
\mc D_{\btt}(n_1,\hdots, n_\ell)\;\coloneqq\; \big\{\mb D\in \mc D(n_1,\hdots,n_\ell): D^{(1)}\btt \cdots \btt D^{(\ell)}\btt D^{(1)}\big\},
$$
and additionally also set
\begin{equation}\label{eq:defDellbtt}
\mc D_\ell(n)\;\coloneqq\;
\mc D(\underbrace{n,\hdots, n}_{\ell \text{ times}}),\quad \mc D_{\btt,\ell}(n)\;\coloneqq\;
\mc D_\btt(\underbrace{n,\hdots, n}_{\ell \text{ times}})\,.
\end{equation}

We stress that the definition of the sets $\mc D_\ell(n)$ and $\mc D_{\btt,\ell}(n)$ as above do not account for possible permutations of dice when checking intransitivity, that is, the list of dice has a fixed order. We will later on deal with {\it random dice}, and because they will have arbitrary distributions the difference between considering or not considering possible permutations will not be important in a relevant way.


Exploring a connection between non-tie dice with integer entries and the set of words in a given alphabet, briefly outlined below and explained in detail in Section~\ref{sec:dicewords}, we will be able to estimate the size of $\mc D_{\btt, \ell}(n)$.
\begin{theorem}\label{thm2}
For each $\ell\geq 3$, there exists a constant $L(\ell)\geq 0$ for which
$$
|\mc D_{\btt,\ell}(n)|\;=\;\ee^{nL(\ell)+o(n)}\quad \text{as }n\to \infty\,.
$$
\end{theorem}

Naturally, we are asked about the value of $L(\ell)$. For any $n\geq 1$, a simple combinatorial argument shows that $|\mc D_{\ell}(n)|=(\ell n)!/(n!)^\ell$, and by combing Theorem~\ref{thm2} with Stirling's approximation we see that
\begin{equation}\label{eq:subexpondecay}
\frac{|\mc D_{\btt,\ell}(n)|}{|\mc D_{\ell}(n)|}\;=\; \ee^{-n(\ell\log \ell-L(\ell))+o(n) } \,.
\end{equation}

Since $\mc D_{\btt,\ell}(n) \subset \mc D_\ell(n)$, we obviously have $L(\ell)\leq \ell\log \ell$. Equipping $\mc D_{\ell}(n)$ with the uniform distribution, one may view the quantity $\frac{|\mc D_{\btt,\ell}(n)|}{|\mc D_\ell(n)|}$ as the probability of selecting an $\ell$-uple of intransitive dice from this distribution. As a consequence of Theorem~\ref{thm:nointransitive} to be seen in a moment, applied to random dice with uniform law on $[0,1]$ we can infer that
$$
\lim_{n\to \infty}\frac{|\mc D_{\btt,\ell}(n)|}{|\mc D_{\ell}(n)|}\;=\;0\,.
$$

It is thus natural to ask whether the decay \eqref{eq:subexpondecay} is exponential or sub-exponential in $n$, and our next main result in fact shows that it is sub-exponential.

\begin{theorem}\label{thm:ellvalue}
For any $\ell\geq 3$, $L(\ell)=\ell\log \ell$.
\end{theorem}

In fact, in a first draft version of this paper, Theorem~\ref{thm:ellvalue} was stated as a conjecture, based on strong numerical evidence performed for the case $\ell=3$. For completeness, in Subsection~\ref{sec:numerics} we discuss such numerical experiments, including an exact calculation of $|\mc D_{\btt,3}(n)|$ for low values of $n$. It remains a rather puzzling question to determine the true rate of decay in \eqref{eq:subexpondecay}.

The proof of Theorem~\ref{thm2} relies on two main ingredients. The first ingredient is a natural and explicit bijection between $\mc D_\ell(n)$ and the set $\mc W_\ell(n)$ consisting of words of length $n\ell$ in an alphabet of $\ell$ letters, where each letter appears $n$ times. Even though this bijection is rather simple, it turns out that properties on intransitive words become more transparent when we move to $\mc W_\ell(n)$. In particular, exploring this connection we establish a certain convexity property on $|\mc D_{\btt,\ell}(n)|$, from which Theorem~\ref{thm2} will then follow from standard arguments. 

However, the convexity argument we just mentioned does not provide any direct information on the value $L(\ell)$. Nevertheless, once again invoking the bijection with $\mc W_\ell(n)$, we establish a procedure to obtain a subset of $\mc D_{\btt, \ell}(n)$ from a large family $\mc Q$ of, in a very precise sense to be specified later, ``almost intransitive'' dice. Each die in $\mc Q$ can be concatenated to a fixed die that is, again in a precise sense, ``highly intransitive'' to generate a new intransitive die. This way, we are able to infer that $|\mc D_{\btt,\ell}(n)|\geq |\mc Q|$.

To complete the proof of Theorem~\ref{thm:ellvalue}, it will remain to find a set $\mc Q$ of low intransitivity for which $\log |\mc Q| \sim {n\ell\log\ell}$. The proof of the existence of such a set is probabilistic and based on our second main ingredient, namely a Central Limit Theorem for random dice, which we discuss next.

\subsection{Main results for random models of dice}\label{subsec:random_dice}\hfill

When each $D_i$ is a random variable, we say that the corresponding die $D=(D_1,\hdots,D_n)$ is a \textbf{random die}. Whenever we say that the \textbf{law} of a die $D$ is $\mc L^D$, we mean that the entries $D_i$ are all i.i.d.\ random variables with law $\mc L^D$.
We say that the dice in a collection $\mb D= (D^{(1)},\hdots, D^{(\ell)})$ are \textbf{independent} if the family of random variables $\{D_i^{(k)}\}$ is mutually independent. We stress that for independent dice the laws $\mc L^{(1)}\coloneqq\mc L^{D^{(1)}},\dots, \mc L^{(\ell)}\coloneqq\mc L^{D^{(\ell)}}$ need not coincide, but entries within the same die are i.i.d. random variables.

Our main goal is to determine whether a sequence of random dice may be intransitive when they grow in size. Fix an integer $\ell\geq 3$ and consider a sequence $\{\mb D_m\}_m$ of collections $\mb D_m=(D^{(1)},\hdots, D^{(\ell)})$ of random independent dice. Each die $D^{(k)}=D^{(k)}(m)$ depends on the index $m$ of the sequence $\{\mb D_m\}_m$, but to lighten notation we mostly omit this dependence. We assume each die $D^{(k)}=D^{(k)}(m)$ has $n_k=n_k(m)\leq m$ faces, which may vary with $m$, and we set
\begin{equation}\label{deff:fkm}
f_k=f_k(m)\;\coloneqq\;\frac{n_k}{m}\quad \text{so that } D^{(k)} \text{ has size } n_k\;=\;f_km,\quad k=1,\hdots, \ell.
\end{equation}
The assumption $n_k=n_k(m)\leq m$ is made solely for convenience as in this case $f_k\leq 1$ for every $k$. Although there are no further relations imposed between the sizes $n_k$ and $m$, it is instructive to think about $m$ as essentially giving the size of the die with largest number of faces.  As we already mentioned before, we always assume that different entries of the same die $D^{(k)}$ are independent random variables with the same law $\mc L^{(k)}=\mc L^{D^{(k)}}$ which may now vary with $m$, and we write $\mc L^{(k)}_m$ when we want to emphasize this dependence.

The main question we investigate in this second part is on how to estimate the probability of intransitivity, namely
\begin{equation}
\label{eq:def_intransitivity_event}
\bb P\big({D^{(1)} \btt D^{(2)} \btt \cdots \btt D^{(\ell)} \btt D^{(1)}}\big),
\end{equation}
as the number of faces of our dice go to infinity, which we measure by sending $m\to\infty$.
The intransitivity event in~\eqref{eq:def_intransitivity_event} is the
intersection of $D^{(k)} \btt D^{(k+1)}$ for $1 \le k \le \ell$, where we convention for the rest of the paper that $D^{(0)} = D^{(\ell)}$ and $D^{(\ell+1)} = D^{(1)}$. Such events are intimately connected to the values of the random variables
\begin{equation}
\label{eq:def_Nk}
N_{k} \;\coloneqq\; \sum_{i=1}^{n_k}\sum_{j=1}^{n_{k+1}}
    \ind_{{{D}_i^{(k)}>D_{j}^{(k+1)}}},\quad k=1,\dots, \ell
\end{equation}
and
\begin{equation}
\label{eq:def_Ek}
E_k\;\coloneqq\; \sum_{i=1}^{n_k}\sum_{j=1}^{n_{k+1}}
    \ind_{{{D}_i^{(k)}=D_{j}^{(k+1)}}},\quad k=1,\dots, \ell.
\end{equation}
From the inequality \eqref{eq:fundcomparison} we learn that $D^{(k)}\btt D^{(k+1)}$ if, and only if, the inequality
\begin{equation*}
N_k\;>\;\frac{1}{2}n_kn_{k+1}-\frac{1}{2}E_k
\end{equation*}
is satisfied, and therefore
\begin{equation}\label{eq:ineqNkEk}
\bb P\left({D^{(1)} \btt \cdots \btt D^{(\ell)} \btt D^{(1)}}\right)\;=\;
\bb P\Big(N_k>\frac{1}{2}n_kn_{k+1}-\frac{1}{2}E_k\,,\; k=1,\hdots,\ell\Big),
\end{equation}
which will be at the core of our method to analyze \eqref{eq:def_intransitivity_event}, and shows the relevance of $N_k$ and $E_k$. One should view the $N_k$ as the {\it relative strength} of the die $D^{(k)}$ against $D^{(k+1)}$.
Observe that for dice coming from a sequence $\{\mb D_m\}_m$, the random variables $N_k=N_k(m)$ and $E_k=E_k(m)$ also depend on $m$, and $N_{k}$, $N_{k+1}$, $E_k$ and $E_{k+1}$ are all pairwise strongly correlated.

We will analyze \eqref{eq:ineqNkEk} in the limit $m\to \infty$ via a Central Limit Theorem (CLT) for the vector $(N_1, \ldots,
N_\ell)$. For this CLT some probabilities associated to the underlying laws of the dice are of utmost importance. By
\begin{equation}\label{eq:defpk}
{\bf p}_k\;=\;{\bf p}(\mc L^{(k)},\mc L^{(k+1)}) \;\coloneqq\; \bb P\left(D^{(k)}_1> D^{(k+1)}_1\right)\;=\;\bb E\left(\ind_{D_1^{(k)}>D_1^{(k+1)}}\right)
\end{equation}
we denote the probability that a given face of the $k$-th die beats a given
face of the  $(k+1)$-th die. By
\begin{equation}\label{eq:defqk}
{\bf q}_k\;=\;{\bf q}(\mc L^{(k)},\mc L^{(k+1)}) \;\coloneqq\; \bb P\left(D^{(k)}_1> D^{(k+1)}_1, D^{(k)}_2> D^{(k+1)}_1\right)
\end{equation}
we denote the probability that two given faces of the $k$-th die beat a given face of the  $(k+1)$-th die. By
\begin{equation}\label{eq:defrk}
{\bf r}_k\;=\;{\bf r}(\mc L^{(k)},\mc L^{(k+1)}) \;\coloneqq\; \bb P\left(D^{(k)}_1> D^{(k+1)}_1, D^{(k)}_1> D^{(k+1)}_2\right)
\end{equation}
we denote the probability that a given face of the $k$-th die beats two given faces of the  $(k+1)$-th die. Finally, also set
\begin{equation}\label{eq:defsk}
{\bf s}_k \;=\;{\bf s}(\mc L^{(k-1)},\mc L^{(k)},\mc L^{(k+1)})\;\coloneqq\; \bb P\left(D^{(k-1)}_1 > D^{(k)}_1 > D^{(k+1)}_1\right),
\end{equation}
which is the probability that a given face of $D^{(k-1)}$ beats a given face of $D^{(k)}$ at the same time that the latter beats a given face of $D^{(k+1)}$.  As we will see in a moment, these quantities will play a role in understanding the covariance between different dice. We use cyclic notation for these quantities, so that $\mb p_{\ell+1}\coloneqq \mb p_1$, $\mb q_{\ell+1}\coloneqq\mb q_1$ and so forth.

As said, our main tool to analyze the probability \eqref{eq:def_intransitivity_event}, and also to prove Theorem~\ref{thm:ellvalue}, is a CLT for the correlated random variables $N_1,\hdots,N_\ell$, so it is natural to introduce their normalized version
\begin{equation}
\label{eq:def_tilde_Nk}
\wt N_k \;\coloneqq\; \frac{N_k - \bb E(N_k)}{\sqrt{\var{N_k}}}\,.
\end{equation}
Let
\begin{align}
\sigma_k
    \;=\;\sigma_k(\mb p_k,\mb q_k,\mb r_k,\mb s_k)
\;\coloneqq\; \Big[f_k f_{k+1}
    \left(f_k(\mb q_k - \mb p_k^2) + f_{k+1}(\mb r_k - \mb p_k^2)\right)
    \Big]^{1/2}\label{eq:def_sigma_k}
\end{align}
and
\begin{equation}
\label{eq:def_gamma_k}
\gamma_k\;\coloneqq\; \frac{1}{\sigma_{k-1}\sigma_k}f_{k-1}f_kf_{k+1}(\mb s_k-\mb p_{k-1}\mb p_k)\,.
\end{equation}
A straightforward calculation (see Lemma~\ref{lema:mean_var_Nk}) shows that, as $m\to \infty$,
\begin{equation}\label{eq:ENkvarNk}
\begin{split}
\bb E (N_k)& \;=\; f_kf_{k+1}m^2\mb p_k\,,  \\
\var{N_k} & \;=\; \sigma_k^2 m^3+o(m^{3})\,,\quad \text{and} \\
\corr{N_{k-1},N_{k}}& \;=\; \gamma_k +o(1)\,.
\end{split}
\end{equation}
We stress that the values $\sigma_k=\sigma_k(m)$ and $\gamma_k=\gamma_k(m)$ depend explicitly on probabilities associated to the laws $\mc L^{(k-1)}_m,\mc L^{(k)}_m$ and $\mc L^{(k+1)}_m$. Moreover, they are $O(1)$ as $m\to \infty$ regardless of the regularity features of these laws, such as whether they have finite moments or how their tails behave.

Since we are considering a sequence $\{\mb D_m\}_m$ of collections of independent dice, all the quantities we just introduced depend on $m$, and when needed to stress such dependence, we write $\mb p_k=\mb p_k(m), \sigma_k=\sigma_k(m),\gamma_k=\gamma_k(m)$ etc.
Our main working assumptions are the following.
\begin{assumption}\label{assumption:main}
Fix $\ell\geq 3$. We assume that the sequence $\{\mb D_m\}_m$ is a collection of $\ell$ independent random dice, each with number of faces $n_k=f_km$ as in \eqref{deff:fkm}, and satisfying the following  conditions:
\begin{enumerate}[(i)]
\item For $k=1,\hdots,\ell$, the relative sizes $f_k=f_k(m)$ satisfy
$$
f_k(m)\;\to\; f_k(\infty)\in (0,1]\,, \quad \text{as }m\to\infty.
$$

\item For $k=1,\hdots, \ell$, the rate of growth of the mean and variance of $N_k$, and covariance between $N_{k-1}$ and $N_k$ satisfy
\begin{align*}
& \mb p_k(m)\;\to\; \mb p_k(\infty)\in (0,1], \\
& \sigma_k(m)\;\to\; \sigma_k(\infty)\in (0,\infty)\,, \\
& \gamma_k(m)\;\to\; \gamma_k(\infty)\in [-1,1]\,,
\end{align*}
as $m\to\infty$.
\end{enumerate}
\end{assumption}

When, for a given $m$, all the laws $\mc L^{(m)}_1,\hdots \mc L^{(m)}_\ell$, are given by a law $\mc L$ without mass points, the values
$$
\mb p_k=\frac{1}{2},\quad \mb q_k=\mb r_k=\frac{1}{3} \quad \text{and}\quad \mb s_k=\frac{1}{6}
$$
are computed somewhat easily, see Lemma~\ref{lem:prqscontdistr} below.

We insist that the values $\mb p_k(m)$ and $\sigma_k=\sigma_k(m)$ depend only on probabilities associated with the underlying laws rather than on qualitative features of them. In particular, Assumption~\ref{assumption:main}--(i) is solely a non-degeneracy condition, which ensures that the number of faces of the dice are all growing, with the same speed $m$ but possibly different rates. With \eqref{eq:defpk} in mind, condition (ii) on $\mb p_k$ essentially says that the limiting laws do not reduce to a deterministic situation where intransitivity does not occur by degeneration. Also, as we said earlier, under Assumption~\ref{assumption:main}--(i) the values $\sigma_k=\sigma_k(m)$ are bounded functions of $m$. Thus, with \eqref{eq:ENkvarNk} in mind, the second condition in (ii) says that the variance in the relative strength of consecutive dice is growing at true speed $m^3$ and not slower. The quantities $\gamma_k(m)$ are correlation coefficients, so they are always bounded, and the third convergence condition in (ii) can always be achieved with a replacement of the original sequence of dice $\{\mb D_m\}_m$ by a subsequence of it.

\begin{theorem}\label{teo:CLT}
Fix $\ell\geq 3$ and for each $m$ let $\mb D_m=(D^{(1)}(m),\hdots, D^{(\ell)}(m))$ be a collection of random independent dice, for which $\{\mb D_m\}_m$ satisfies Assumption~\ref{assumption:main}, and let $(\widetilde N_1(m),\ldots, \widetilde N_\ell(m))$ be the corresponding variables from \eqref{eq:def_tilde_Nk}.

Then, as $m\to\infty$, the random vector $(\widetilde N_1(m),\cdots, \widetilde N_\ell(m))$ converges in distribution to  a centered Gaussian vector $(X_1, \ldots, X_{\ell})$
whose covariance matrix is given by
\begin{equation}\label{eq:covariancematrixCLT}
\Sigma \;=\;\left(
\begin{array}{cccccccc}
1 & \gamma_2(\infty) & 0 & \cdots  & 0 &\gamma_1(\infty)\\
\gamma_2(\infty) & 1 & \gamma_3(\infty) & \cdots & 0 & 0\\
0 & \gamma_3(\infty) & 1 & \cdots& 0 & 0\\
\vdots & \vdots & \vdots & \ddots & \vdots & \vdots\\
0 & 0 & 0 & \cdots & 1 & \gamma_{\ell}(\infty)\\
\gamma_1(\infty) & 0 & 0 & \cdots & \gamma_{\ell}(\infty) & 1
\end{array}\right),
\end{equation}
where the coefficients $\gamma_k(\infty)$ are the ones in Assumption~\ref{assumption:main}--(ii).
\end{theorem}

Theorem~\ref{teo:CLT} is an analogue for unconstrained dice of \cite[Theorem~1.4]{PTRF2} which considers two uniform models of dice with constrained sum of faces.

As explained in \eqref{eq:subexpondecay} \textit{et seq.}, we will use Theorem~\ref{teo:CLT} to construct a particular set $\mc Q$ consisting of collections of dice which are almost intransitive but not necessarily intransitive, and exploring such set we will be able to prove Theorem~\ref{thm:ellvalue}. 

As a second application, we now discuss how Theorem~\ref{teo:CLT} may be
applied to show that a collection of random dice is not intransitive,
asymptotically almost surely, as $m\to\infty$.

In general, the very definition of $\mb p$ in \eqref{eq:defpk} would say that
\begin{align}\label{eq:identitypksymm}
    1 & \;=\;\bb P\big(D_1^{(k)}>D_1^{(k+1)}\big)+\bb
        P\big(D_1^{(k)}<D_1^{(k+1)}\big)+\bb
        P\big(D_1^{(k)}=D_1^{(k+1)}\big)\nonumber\\
    & \;=\;  \mb p(\mc L^{(k)},\mc L^{(k+1)})+ \mb p(\mc L^{(k+1)},\mc L^{(k)})+\bb P\big(D_1^{(k)}=D_1^{(k+1)}\big)\,.
\end{align}
In order for the die $D^{(k)}$ to be sufficiently stronger than the die $D^{(k+1)}$, we would expect that $\mb p(\mc L^{(k)},\mc L^{(k+1)})>\mb p(\mc L^{(k+1)},\mc L^{(k)})$, and in such a case we would expect $D^{(k)} \btt D^{(k+1)}$ with high probability. Likewise, if $\mb p(\mc L^{(k+1)},\mc L^{(k)})<\mb p(\mc L^{(k)},\mc L^{(k+1)})$ then we would instead expect $D^{(k+1)} \btt D^{(k)}$ with high probability. Hence, intransitivity becomes a nontrivial question precisely when $\mb p(\mc L^{(k)},\mc L^{(k+1)})\approx \mb p(\mc L^{(k+1)},\mc L^{(k)})$ asymptotically as $m\to\infty$, in which case the equality above becomes
$$
\mb p(\mc L^{(k+1)},\mc L^{(k)})+\frac{1}{2}\bb P\big(D_1^{(k)}=D_1^{(k+1)}\big)=\mb p_k+\frac{1}{2}\bb P\big(D_1^{(k)}=D_1^{(k+1)}\big)\;\approx\; \frac{1}{2}\,,
$$
and our next result gives a rate of decay of such approximation under which we can use the CLT to estimate the probability of intransitivity in the large-dice limit $m\to\infty$.

\begin{theorem}\label{thm:intransitive}
    Fix $\ell\geq 3$ and for each $m$ let $\mb D_m=(D^{(1)}(m),\hdots, D^{(\ell)}(m))$ be a collection of random independent dice, for which $\{\mb D_m\}_m$ satisfies Assumption~\ref{assumption:main}, and let $(X_1,\hdots,X_\ell)$ be a Gaussian vector with covariance matrix \eqref{eq:covariancematrixCLT}. Suppose that 
    \begin{equation}\label{eq:asymptoticnotie}
        \lim_{m\to\infty} \bb P\left(D^{(k)}_1(m)=D_1^{(k+1)}(m)\right)\;=\; 0\quad \text{for } k \in [\ell]
    \end{equation}
    holds for every sufficiently large $m$, and that the estimate
    \begin{equation}\label{eq:condttslowgrowthpk_lim}
        \frac{1}{2}-\mb p_k-\frac{1}{2}\bb P(D_1^{(k)}=D_1^{(k+1)})
         =o(m^{-1/2}),\quad m\to +\infty,
    \end{equation}
    holds for  $k \in [\ell]$. Then, it holds
    \begin{align}
        \bb P\left(X_j> 0,\; j \in [\ell] \right)
        &\le \liminf_{m\to \infty}\,\bb P\left({D^{(1)} \btt \cdots \btt D^{(\ell)} \btt D^{(1)}}\right) \nonumber\\
        &\le \limsup_{m\to \infty}\,\bb P\left({D^{(1)} \btt \cdots \btt D^{(\ell)} \btt D^{(1)}}\right)
        \le \bb P\left(X_j\geq 0,\; j \in [\ell] \right).
        \label{eq:intransitivelimit_ineq}
    \end{align}
    In particular, when the Gaussian vector $(X_1, \hdots, X_\ell)$ is non-degenerate we have
    \begin{equation}\label{eq:intransitivelimit_lim}
        \lim_{m\to \infty}\,\bb P\left({D^{(1)} \btt \cdots \btt D^{(\ell)} \btt D^{(1)}}\right)
        = \bb P\left(X_j\geq 0,\; j \in [\ell] \right).
    \end{equation}
\end{theorem}

\begin{remark}
\label{rem:asymptotic_no_tie}
When working with no-tie collections of dice, condition~\eqref{eq:asymptoticnotie}
is immediate. Another simple situation is when all the laws are the
same $\mc L^{(1)}=\cdots= \mc L^{(\ell)}$, since \eqref{eq:identitypksymm}
implies
\begin{equation*}
\frac{1}{2}-\mb p_k-\frac{1}{2}\bb P\left(D_1^{(k)}=D_1^{(k+1)}\right)=0,\quad k=1,\hdots, \ell,
\end{equation*}
so condition~\eqref{eq:condttslowgrowthpk_lim} holds.
In order to evaluate the Gaussian probability in~\eqref{eq:intransitivelimit_lim} we need to understand
better the covariance matrix $\Sigma$ in \eqref{eq:covariancematrixCLT}, see
for instance Proposition~\ref{prop:h2_eigenspace_zero}.
Applications are given in Section~\ref{sec:examples}.
\end{remark}

%

Note that \eqref{eq:asymptoticnotie} says that there are no ties between different dice in the asymptotic limit. Similarly as for $N_k$, the mean and variance of the $E_k$'s are given in terms of probabilities associated to the underlying laws. For arbitrary underlying laws of the entries of the dice, they satisfy the rough bound
\begin{equation}\label{eq:variancetiesgeneral}
\bb E (E_k)=O(m^2)\quad \text{and}\quad \var{E_k}=O(m^3)\quad \text{as } m\to\infty,
\end{equation}
see Lemma~\ref{lem:meanvarEk} below. These quantities have the same order as the corresponding quantities for $N_k$ (compare \eqref{eq:ENkvarNk} with \eqref{eq:variancetiesgeneral}). Lemma~\ref{lem:notiemeanvariance} below shows that \eqref{eq:asymptoticnotie} implies
\begin{equation}\label{eq:variancetiesslow}
\bb E (E_k)=o(m^2),\quad \var{E_k}=o(m^3)\quad \text{as } m\to\infty, \; \text{ for }k=1,\hdots, \ell.
\end{equation}
Thus, condition~\eqref{eq:asymptoticnotie} in Theorem~\ref{thm:intransitive}
may also be interpreted as saying that whenever the $E_k$'s grow slightly
slower than $N_k$'s, either in their mean or in their variance, then the
intransitive dice problem can be bounded from above by the Gaussian
probabilities \eqref{eq:intransitivelimit_limsup} and
\eqref{eq:intransitivelimit_liminf}. 

Even though the covariance matrix \eqref{eq:covariancematrixCLT} is structured,
computing the probability on Equation~\eqref{eq:intransitivelimit_lim} for general $\gamma_k(\infty)$'s explicitly is a challenge. Nevertheless, when $\mb p_k,\mb q_k,\mb r_k$ and $\mb s_k$ are asymptotically the same as for i.i.d. continuous dice, we are able to show that this probability is zero, obtaining the next result.

\begin{theorem}\label{thm:nointransitive}
Let $\{\mb D_m\}_m$ be a sequence of random independent dice satisfying the
conditions of Theorem~\ref{thm:intransitive}. In addition, suppose that
\begin{equation}\label{eq:assumptionpqrssym}
\mb p_k\to \frac{1}{2},\quad \mb q_k \to \frac{1}{3}, \quad \mb r_k\to \frac{1}{3} \quad \text{and}\quad \mb s_k\to \frac{1}{6} \quad \text{as }m \to \infty,
\end{equation}
for every $k=1,\hdots, \ell$. Then
$$
\lim_{m\to \infty}\bb P\left({D^{(1)} \btt \cdots \btt D^{(\ell)} \btt D^{(1)}}\right)\;=\;0\,.
$$
\end{theorem}
In words, under the assumptions of Theorem~\ref{thm:nointransitive} above, the proportion of intransitive random independent dice becomes negligible as the number of faces grows.
In Section~\ref{sec:examples} we also give an example of 3 dice with different laws for which such proportion converges to a value in $(0,1)$.

For i.i.d. dice with the same number of faces, Theorem~\ref{thm:nointransitive} was obtained first by H\k{a}z{\l}a, Mossel, Ross and Zheng \cite[Theorem~6]{PTRF}. We decided to present it here as it is another application of Theorem~\ref{teo:CLT}, and it applies to a larger setup.

Both Theorem~\ref{thm:ellvalue} and Theorem~\ref{thm:nointransitive} rely on Theorem~\ref{teo:CLT}, but on its own the proof of Theorem~\ref{teo:CLT} is the most involved proof in this paper. We approach it via the moment method. The quantities $\wt N_k$'s are correlated Bernoulli random variables, and to control their moments we rely on combinatorial arguments. We expand such moments in powers of $m$, and relate each such coefficient to the number of graphs with some particular properties. By a careful estimation of the number of such graphs, we are then able to match the large $m$ limit of such moments with the claimed Gaussian.

The proof of Theorem~\ref{thm:intransitive} is based on Theorem~\ref{teo:CLT}.
We look at the probability of $(D^{(1)}, \ldots, D^{(\ell)})$ forming an
intransitive cycle, and with the help of Chebyshev's inequality we condition on the event of no ties, reducing the right-hand side of \eqref{eq:ineqNkEk} to a probability that involves only the $N_k$'s plus an additional term which is small in virtue of the variance control \eqref{eq:variancetiesslow}. The right-hand side then naturally arises when taking the large $m$ limit.
In virtue of a particular structure of the coefficients $\gamma_k(\infty)$ in
the covariance matrix \eqref{eq:covariancematrixCLT}, we are able to show that
the probability on the right-hand side of \eqref{eq:intransitivelimit_lim} vanishes, and Theorem~\ref{thm:nointransitive} follows.

\subsection{Organization of the remainder of the paper}\hfill

 The remainder of the paper is structured as follows.
 In Section~\ref{sec:examples} we discuss examples of random dice,  in particular  when our core Assumption~\ref{assumption:main} and the variance control \eqref{eq:variancetiesslow} are satisfied, allowing to apply our main result. We also provide a sequence of random independent dice that do not satisfy these conditions and for which intransitivity survives in the limit.
  In Section~\ref{sec:deterministic} we discuss intransitivity in deterministic contexts, and in particular we explore a connection between intransitive dice and combinatorics of words in order to construct intransitive dice.
  We then turn to the context of random dice. In Section~\ref{sec:countingGaussian} we briefly discuss the counting functions $N_k$ and $E_k$ from \eqref{eq:def_Nk}--\eqref{eq:def_Ek}, which correspond to victories and ties, respectively, and which play a central role in the connection between our CLT and intransitivity. 
In our CLT, Gaussian vectors with a covariance matrix of a particular structure appear (see \eqref{eq:covariancematrixCLT}), and in Section~\ref{sec:countingGaussian} we also collect several properties of them in a form suitable for our needs.
  In Section~\ref{s:thms5_and_6}, we assume Theorem~\ref{teo:CLT}, which is our central limit theorem, and we use it to prove Theorems~\ref{thm:intransitive} and \ref{thm:nointransitive}, which are tests of asymptotic intransitivity. Finally, in Section~\ref{sec:proof_CLT}, we prove Theorem~\ref{teo:CLT}, and in Section~\ref{sec:stringsAsymptotics} we use several outputs obtained in previous sections to conclude the proof of Theorem~\ref{thm:ellvalue}.

\subsection*{Acknowledgments}
T.F.\ acknowledges support by the National Council for Scientific and Technological Development (CNPq) via a Universal Grant (number 406001/2021-9) and a Bolsa de Produtividade (number 311894/2021-6). 

G.S.\ acknowledges support by São Paulo Research Foundation (FAPESP) under Grants \# 2019/16062-1 and \# 2020/02506-2, and by Brazilian National Council for Scientific and Technological Development (CNPq) under Grant \# 315256/2020-6.

J.P.\ acknowledges support by São Paulo Research Foundation (FAPESP) under Grant \# 2023/02674-0.

J.Pa.\ acknowledges support by Brazilian National Council for Scientific and Technological Development (CNPq) under Grant \# 118536/2023-0

L.C.\ acknowledges support by São Paulo Research Foundation (FAPESP) under Grant
\# 2023/02240-0.

L.L.\ acknowledges support by São Paulo Research Foundation (FAPESP) under Grant
\# 2023/02397-7.

Part of this work was carried out during the undergraduate research program "Jornadas de Pesquisa em Matemática do ICMC 2023" held at the Instituto de Ciências Matemáticas e de Computação (ICMC) - Universidade de São Paulo (USP), and which was partially supported by the Centro de Ciências Matemáticas Aplicadas à Indústria (CeMEAI - CEPID) under FAPESP Grant \# 2013/07375-0. Research carried out using the computational resources of the Center for Mathematical Sciences Applied to Industry (CeMEAI) funded by FAPESP (grant 2013/07375-0). We thank the hospitality of ICMC-USP during the program.

\section{Examples}\label{sec:examples}

In this section we describe  examples of random dice  such that the probability
of observing intransitivity is asymptotically null by applying the
Theorem~\ref{thm:nointransitive} and we also illustrate some cases of
asymptotically intransitive random dice.

\subsection{Dice with same laws}\hfill

We start by recalling that when all laws are the same, $\mc L^{(1)}=\cdots= \mc
L^{(\ell)}$, then condition~\eqref{eq:condttslowgrowthpk_lim} always holds, see Remark~\ref{rem:asymptotic_no_tie}.
The first example has been already commented below Theorem~\ref{teo:CLT}: assuming that each die has the same number of faces, and those faces are i.i.d.\ random variables with the same continuous (but not necessarily absolutely continuous) law $\mc L$, the probability that the random dice $(D^{(1)},\ldots,D^{(\ell)})$ are intransitive goes to zero as $m\to \infty$.
Theorem~\ref{thm:nointransitive} straightforwardly extends this to a  more general situation, as we explain in the next paragraph.

If the law of any die is given by a same continuous  law $\mc L$, there will be no ties, so \eqref{eq:asymptoticnotie} holds trivially. Moreover, $\mb p_k(m), \sigma_k(m)$ and $\gamma_k(m)$ do not depend on $m$ and neither on $k$, hence it is trivial to check Assumption~\ref{assumption:main}-(ii). Assuming that the quantity of faces in the $k$-th die is given by $n_k=f_km$, where each $f_k$ is a positive constant, we verify   Assumption~\ref{assumption:main}-(i). These conditions together lead to the conclusion, by Theorem~\ref{thm:nointransitive}, that the sequence of dice constructed in this way has asymptotically null probability of being intransitive. That is, under a continuous law, intransitivity is not achievable regardless of the quantities of faces in each die, provided these quantities are proportional to the scaling parameter $m$.

Let us see now a discrete example. Assume that all $\ell$ dice have same law $\mc L_m$, the law of a geometric random variable of parameter $p$. Since
\begin{equation*}
\bb P\left(D^{(k)}_1(m)=D_1^{(k+1)}(m)\right)\;=\; \sum_{i=1}^\infty (1-p)^{2(i-1)}p^{2} \;=\;\frac{p}{2-p}\,,
\end{equation*}
in order to ensure  condition \eqref{eq:asymptoticnotie} on ties, it is necessary to impose that $p=p(m)\to 0$ as $m\to\infty$. In this case, a long but elementary calculation yields
\begin{align*}
\mb p_k(m)&\;=\;\frac{1-p}{2-p}&&\longrightarrow \frac{1}{2}\,, \\
\mb q_k(m)&\;=\;\frac{(1-p)^2}{3-3p+p^2} &&\longrightarrow \frac{1}{3}\,,\\
\mb r_k(m)&\;=\;\frac{(1-p)(2-2p+p^2)}{(2-p)(3-3p+p^2)} &&\longrightarrow \frac{1}{3}\,,\\
\mb s_k(m)&\;=\; \frac{(1-p)^3}{(2-p)(3-3p+p^2)} &&\longrightarrow \frac{1}{6}.
\end{align*}
Recalling the formulas \eqref{eq:def_sigma_k} and \eqref{eq:def_gamma_k} for $\sigma_k$ and $\gamma_k$, respectively, the previous computations give us that, as $m\to\infty$,
\begin{align*}
& \sigma_k(m)\;\longrightarrow\;\sigma(\infty)= \sqrt{\frac{f_kf_{k+1}(f_k+f_{k+1})}{12}}\in (0,\infty)\,, \\
    & \gamma_k(m)\;\longrightarrow\;\gamma(\infty)
    = -\frac{f_{k-1}f_kf_{k+1}}{\sqrt{f_{k-1}f_k(f_{k-1}+f_{k})}\sqrt{f_{k}f_{k+1}(f_{k}+f_{k+1})}}
    \in (-1,0)\,, (\text{see Proposition~\ref{prop:propertiescoeffgammak}})
\end{align*}
so all assumptions of  Theorem~\ref{thm:nointransitive} have been checked, hence the probability of observing a sequence of  intransitive dice is asymptotically null. 

\subsection{Dice with different laws}
\label{sub:dice_with_different_laws}\hfill 

Another interesting application of Theorems~\ref{teo:CLT},
\ref{thm:intransitive} and~\ref{thm:nointransitive} is the possibility of
allowing different laws for each die. It is not straightforward to make a
thorough exploration of such possibility, since we have much freedom of choice.
However, a setup that allows for general results comes from the idea of
building random dice from a known initial collection of deterministic dice, as
follows:
\begin{definition}
\label{defi:blowupseq}
Given a deterministic set of dice $A^{(k)}$ with $k \in [\ell]$ and $m$ faces
each, define its \textbf{blow up sequence} as the sequence of random
dice $B^{(k)}(n)$ with $n$ faces and $k \in [\ell]$ such that 
the law $\mc L_k$ is uniform over the faces of $A^{(k)}$.
\end{definition}
An advantage of working with blow up sequences is that parameters $\mb p_k, \mb
r_k, \mb q_k, \mb s_k$ are constants depending only on the original set of dice
$\{A^{(k)}\}$. Indeed, if the faces of die $A^{(k)}$ are given by $A^{(k)} = (a^{(k)}_1, \ldots, a^{(k)}_{m})$ we have:
\begin{equation}
    \label{eq:blowup_coefficients}
    \begin{aligned}
    \mb p_k
        &= \mathbb P(B_1^{(k)} > B_1^{(k+1)})
        &&= m^{-2} \cdot |\{(i,j); a_i^{(k)} > a_j^{(k+1)}\}|;\\
    \mb q_k
        &= \mathbb P(B_1^{(k)} > B_1^{(k+1)}, B_2^{(k)} > B_1^{(k+1)})
        &&= m^{-3} \cdot |\{(i,j,z); a_i^{(k)} > a_z^{(k+1)}, a_j^{(k)} >
        a_z^{(k+1)}\}|;\\
    \mb r_k
        &= \mathbb P(B_1^{(k)} > B_1^{(k+1)}, B_1^{(k)} > B_2^{(k+1)})
        &&= m^{-3} \cdot |\{(i,j,z); a_i^{(k)} > a_j^{(k+1)}, a_i^{(k)} >
        a_z^{(k+1)}\}|;\\
    \mb s_k
        &= \bb P(B^{(k-1)}_1 > B^{(k)}_1 > B^{(k+1)}_1),
        &&= m^{-3} \cdot |\{(i,j,z); a_i^{(k-1)} > a_j^{(k)} > a_z^{(k+1)}\}|.
    \end{aligned}
\end{equation}
For instance, the property $A^{(k)} \btt A^{(k+1)}$ is equivalent to 
\begin{equation*}
\mb p_k
    = \mathbb P(B_1^{(k)} > B_1^{(k+1)})
    = \mathbb P\bigl(\text{Unif}\{A^{(k)}\} > \text{Unif}\{A^{(k+1)}\}\bigr)
    > \mathbb P\bigl(\text{Unif}\{A^{(k+1)}\} > \text{Unif}\{A^{(k)}\}\bigr)
\end{equation*}
and when there are no ties, this simplifies to 
$$\mb p_k = \mathbb P\bigl(\text{Unif}\{A^{(k)}\} > \text{Unif}\{A^{(k+1)}\}\bigr) > 1/2.$$
The next proposition shows that starting from a deterministic set of
intransitive dice leads to the probability of intransitivity converging
to 1 exponentially fast.
\begin{proposition}
	\label{prop:asymp_intransitive_example}
	Let $A^{(k)}$ for $k \in [\ell]$ be a set of $\ell$ deterministic dice
    with $m$ faces that is known to be intransitive:
	$A^{(k)} \btt A^{(k+1)}$ for every $k$. Consider its blow up sequence
    $(B^{(k)}(n): k \in [\ell])$ as in Definition~\ref{defi:blowupseq}.
    Then, there is a constant $c > 0$, depending
	only on the set of dice $A^{(k)}$, such that
	\begin{equation}
	\label{eq:asymp_intransitive_example}
	\bb P\big(B^{(1)} \btt B^{(2)} \btt \cdots \btt B^{(\ell)} \btt B^{(1)}\big)
	\;=\; 1 + o(e^{-cn})\,, \quad \text{as $n \to \infty$}.
	\end{equation}
\end{proposition}

\begin{proof}
	The deterministic dice $A^{(k)}$, $k \in [\ell]$ form an intransitive cycle. This
	can be translated into the following collection of inequalities:
	\begin{equation}
	\label{eq:dice_Ak_condition}
	\sum_{\substack{i,j:\\ a^{(k)}_i > a^{(k+1)}_j}} 1
	\;>\; \sum_{\substack{i,j:\\ a^{(k)}_i < a^{(k+1)}_j}} 1\,,
	\quad \text{for every $k \in [\ell]$}.
	\end{equation}
	Let $N_{k,i}$ be the random variable that counts the number of appearances of
	face $a^{(k)}_i$ at dice $B^{(k)}$. As discussed in Section~\ref{sec:deterministic},
	the quantity $N_{k,i}$ represents the weight of face $a^{(k)}_i$ in die
	$B^{(k)}$. Hence, we can write event $B^{(k)} \btt B^{(k+1)}$ as a function
	of $N_{k,i}$ and $N_{k+1,j}$:
	\begin{equation}
	\sum_{\substack{i,j:\\ a^{(k)}_i > a^{(k+1)}_j}} N_{k,i} N_{k+1,j}
	\;>\; \sum_{\substack{i,j:\\ a^{(k)}_i < a^{(k+1)}_j}} N_{k,i} N_{k+1,j}\,,
	\quad \text{for every $k \in [\ell]$}.
	\end{equation}
	It is clear that $N_{k,i}$ has a binomial distribution with parameters $n$ and
	$\frac{1}{m}$, and by Hoeffding's inequality
	\begin{equation}
	\label{eq:concentration_N_k_i}
	\bb P \Bigl( \Bigl|N_{k,i} - \frac{n}{m}\Bigr| > \varepsilon n\Bigr)
	\;\le\; 2 e^{- 2 \varepsilon^2 n}\,.
	\end{equation}
	Define the event $G := \bigcap_{k,i} \bigl\{
	\bigl|N_{k,i} - \frac{n}{m}\bigr| \le \varepsilon n \bigr\}$. By union bound,
	\begin{equation*}
	\bb P(G^c)
	\;=\; \bb P \Bigl( \bigcup_{k,i} \bigl\{
	\bigl|N_{k,i} - \frac{n}{m}\bigr| > \varepsilon n \bigr\}\Bigr)
	\;\le\; 2\ell m\, e^{- \varepsilon^2 n}\,.
	\end{equation*}
	Notice that on event $G$ we have that
	\begin{equation}
	\label{eq:estimates_prod_Ni_Nj}
	n^2 \Bigl(\frac{1}{m} - \varepsilon \Bigr)^2
	\;<\; N_{k,i} N_{k+1,j}
	\;<\; n^2 \Bigl(\frac{1}{m} + \varepsilon \Bigr)^2.
	\end{equation}
	From~\eqref{eq:dice_Ak_condition} it is clear that
	\begin{equation*}
	\sum_{\substack{i,j:\\ a^{(k)}_i > a^{(k+1)}_j}} 1
	 \;-\;  \sum_{\substack{i,j:\\ a^{(k)}_i < a^{(k+1)}_j}} 1
	\;\ge\; 1\,,\quad \text{for every $k \in [\ell]$}.
	\end{equation*}
	By continuity, one can choose $\varepsilon > 0$ such that
	\begin{equation*}
	\Bigl(\frac{1}{m} - \varepsilon \Bigr)^2
	\sum_{\mathclap{\substack{i,j:\\ a^{(k)}_i > a^{(k+1)}_j}}} 1
	\quad - \quad \Bigl(\frac{1}{m} + \varepsilon \Bigr)^2
	\sum_{\mathclap{\substack{i,j:\\ a^{(k)}_i < a^{(k+1)}_j}}} 1
	\;>\; \frac{1}{2m^2}
	\;>\;0\,,\quad \text{for every $k \in [\ell]$}.
	\end{equation*}
	Apply the upper estimate of~\eqref{eq:estimates_prod_Ni_Nj} for pairs $i,j$
	with $a^{(k)}_i < a^{(k+1)}_j$ and the lower estimate for pairs with
	$a^{(k)}_i > a^{(k+1)}_j$. Then, on the event $G$ we have
	\begin{align*}
	\sum_{\mathclap{\substack{i,j:\\ a^{(k)}_i < a^{(k+1)}_j}}} N_{k,i} N_{k+1,j}
	\;\;<\;\; n^2 \Bigl(\frac{1}{m} + \varepsilon \Bigr)^2
	\sum_{\mathclap{\substack{i,j:\\ a^{(k)}_i < a^{(k+1)}_j}}} 1
	\;\;<\;\; n^2 \Bigl(\frac{1}{m} - \varepsilon \Bigr)^2
	\sum_{\mathclap{\substack{i,j:\\ a^{(k)}_i > a^{(k+1)}_j}}} 1
	\;\;\;<\;\;\;
	\sum_{\mathclap{\substack{i,j:\\ a^{(k)}_i > a^{(k+1)}_j}}} N_{k,i} N_{k+1,j}\,,
	\end{align*}
	implying that on $G$ we have $B^{(k)} \btt B^{(k+1)}$ for every $k \in [\ell]$.
\end{proof}
As an application of above, take for instance the Effron's dice, which are given by
\begin{align*}
&A=(0, 0, 4, 4, 4, 4),\quad
B=(3, 3, 3, 3, 3, 3)\\
&C = (2, 2, 2, 2, 6, 6),\quad
D = (1, 1, 1, 5, 5, 5),
\end{align*}
as mentioned in the Introduction. In this case, the laws associated to  each die would be:
\begin{align*}
&\mc L_A= \frac{1}{3}\delta_0+ \frac{2}{3}\delta_4\,,\quad
\mc L_B=\delta_3\,,\quad \mc L_C = \frac{2}{3}\delta_2 +\frac{1}{3} \delta_6\,,\quad\text{ and }\quad\mc L_D = \frac{1}{2}\delta_1+ \frac{1}{2}\delta_5\,.
\end{align*}

Of course, since the corresponding  sequence of dice  with the above laws is asymptotically intransitive, it cannot fulfill the assumptions of Theorem~\ref{thm:nointransitive}. Note that there are no ties in Effron's dice example, so condition \eqref{eq:asymptoticnotie} is trivially satisfied.
The assumptions of Theorem~\ref{thm:nointransitive} are not satisfied because $\mb p_k > \frac{1}{2}$, which tells us that the quantity of victories of a die against another die becomes deterministic (in view of the law of large numbers) as the number of faces in each die increases.

The examples covered by Proposition~\ref{prop:asymp_intransitive_example} are
in some way straightforward. For instance, if $\mb p_k > \frac{1}{2}$ for every $k \in
[\ell]$ then we have asymptotic intransitivity with probability 1. If the starting
collection of dice $\{A^{(k)}\}$ is `more balanced', then intransitivity is not
as immediate.

\begin{example}
\label{ex:blowup_close_intransitive}
Consider the collection of 3 six-sided dice:
\begin{equation*}
\text{$A = (18,13,11,7,6,2)$,
$B = (17,15,10,8,4,3)$ and $C = (16,14,12,9,5,1)$.}
\end{equation*}
The choice of such a sequence was based on it having a `very symmetrical'
string representation $\mb W = ABCBCACABCBAACBBAC$, see
Section~\ref{sec:dicewords} for more details. Although the blow up sequence
has different laws for each dice, straighforward counting (using the
representation as $\mb W$, for instance) shows that parameters $\mb p_k, \mb
r_k, \mb q_k, \mb s_k$ do not depend on $k$ and are given by
\begin{equation*}
\mb p_k = \frac{1}{2},
    \qquad \mb q_k = \frac{70}{216},
    \qquad \mb r_k = \frac{76}{216},
    \quad \text{and}
    \quad \mb s_k = \frac{1}{6},
\end{equation*}
implying that $\sigma_k = (\frac{19}{108})^{1/2}$ and
$\gamma_k = -\frac{9}{19}$. Finally, computing the determinant
of matrix $\Sigma$, we obtain $\det \Sigma = \frac{784}{6859} > 0$ which implies
asymptotic intransitivity with probability bounded away from zero and one. 
\end{example}

Notice that in Example~\ref{ex:blowup_close_intransitive} we are quite close to
the coefficients $\mb p_k = \frac{1}{2}, \mb q_k = \mb r_k = \frac{1}{3}, \mb
s_k = \frac{1}{6}$, which are obtained for iid continuous distributions
(indeed, notice that $\frac{70}{216} = 0.32\overline{407}$ and
$\frac{76}{216} = 0.35\overline{185}$).
A natural follow up question is if it is possible to find a starting collection
of dice $A^{(k)}$ such that the above coefficients are achieved for the blow up
sequence. We show that this is actually
impossible.

\begin{proposition}
\label{prop:blowup_impossible_nointransitive}
    Let $(A^{(k)}:\ k \in [\ell])$ be a collection of $m$ sided no-tie dice and
    $\mb p_k, \mb q_k, \mb r_k, \mb s_k$ be the coefficients of its blow up
    sequence. If $\mb p_k = \frac{1}{2}$, then
    $\mb q_k+\mb r_k > \frac{2}{3}$. In particular, one of $\mb q_k$ or $\mb r_k$ is strictly greater than $\frac{1}{3}$
\end{proposition}

\begin{proof}
    Fix $k$ such that $\mb p_k = \frac{1}{2}$.
    The coefficient $\mb r_k$ is given by 
    \begin{equation*}
    \mb r_k
        = m^{-3} \cdot
            |\{(i,j,z); a_i^{(k)} > a_j^{(k+1)}, a_i^{(k)} > a_z^{(k+1)}\}|.
    \end{equation*}
    Let us assume that $(a_1^{(k)}, \ldots, a_m^{(k)})$ and 
    $(a_1^{(k+1)}, \ldots, a_m^{(k+1)})$, the faces of $A^{(k)}$ and
    $A^{(k+1)}$, are both in decreasing order, without loss of generality.
    Take this collection of $2m$ values and sort it in decreasing order
    (using that we have no ties) to form a word $\mb W$ with length $2m$
    made of two letters, $A$ if the face value comes from $A^{(k)}$ and
    $B$ if the face value comes from $A^{(k+1)}$.
    For instance, if $A^{(k)} = (6,4,3)$ and $A^{(k+1)} = (5,2,1)$ then
    $\mb W = ABAABB$. For more details on these string representations, see
    Section~\ref{sec:dicewords}.
    
    Let us define $\alpha_i$ as the number of $A$s before the $i$-th $B$ (for $\mb W =
    ABAABB$, we have $\alpha_1=1, \alpha_2=\alpha_3=3$). Then,
    by~\eqref{eq:blowup_coefficients} and the hypotheses on $\mb p_k$
    we have
    \begin{equation*}
        \frac{1}{2} = \mb p_k = \frac{1}{m^2} \sum_{k=1}^m \alpha_k
        \quad \text{and} \quad
        \mb q_k = \frac{1}{m^3} \sum_{k=1}^m \alpha_k^{2}.
    \end{equation*}
    Moreover, if $U_1, U_2, U_3$ are iid.\ uniform random variables on $[m]$ then
    \begin{equation*}
        \mb r_k
        = \bb P \bigl(a^{(k)}_{U_1} > \max\{a^{(k+1)}_{U_2}, a^{(k+1)}_{U_3}\}\bigr)
        = \bb P \bigl(a^{(k)}_{U_1} > a^{(k+1)}_{\min\{U_2,U_3\}}\bigr)
    \end{equation*}
    and we can also interpret this probability in terms of positioning of $A$s and
    $B$s in $\mb W$. Noticing that $\bb P(\min\{U_2,U_3\}=k) = \frac{2(m-k)+1}{m^2}$,
    it follows that
    \begin{align*}
    \mb r_k
        &= \frac{1}{m^3} \sum_{k=1}^m (2(m-k) + 1 )\alpha_k
        = \frac{2m+1}{m^3} \sum_{k=1}^m \alpha_k
            - \frac{2}{m^3} \sum_{k=1}^m k\alpha_k\\
        &= \frac{2m+1}{2m} - \frac{2}{m^3} \sum_{k=1}^m k\alpha_k.
    \end{align*}
    Hence,
    \begin{align*}
        \mb q_k+\mb r_k&=\frac{2m+1}{2m} + \frac{1}{m^3} \sum_{k=1}^m (k-\alpha_k)^2-\frac{1}{m^3}\sum_{k=1}^{m}k^2\\
        &=\frac{2}{3}-\frac{1}{6m^2}+\frac{1}{m^3} \sum_{k=1}^m (k-\alpha_k)^2.
    \end{align*}
    
    Invoking once again the hypothesis over $\mb p_k$ and the QM-AM inequality,
    \begin{equation*}
        \frac{1}{m^3} \sum_{k=1}^m (k-\alpha_k)^2 
        \ge \frac{1}{m^2} \left(\frac{1}{m}\sum_{k=1}^m (k-\alpha_k)\right)^2\\
        = \frac{1}{4m^2},
    \end{equation*}
    and the result follows.
\end{proof}

As a final example, we notice that although we cannot recover the coefficients $\mb p_k = \frac{1}{2}, \mb q_k = \mb r_k = \frac{1}{3}, \mb
s_k = \frac{1}{6}$ using blow up sequences, they are still able to provide examples with probability of an intransitive cycle going to zero, when we have different laws.
\begin{example}
\label{ex:blowup_not_intransitive}
Consider the collection of 3 six-sided dice:
\begin{equation*}
\text{$A = (15,14,13,6,5,4)$,
$B = (18,17,16,3,2,1)$ and $C = (12,11,10,9,8,7)$,}
\end{equation*}
whose string representation is $\mb W = BBBAAACCCCCCAAABBB$. Parameters $\mb p_k, \mb r_k, \mb q_k, \mb s_k$ are given by
\begin{align*}
\mb p_1 &= \tfrac{1}{2}, & \mb p_2&= \tfrac{1}{2} & \mb p_3&= \tfrac{1}{2} \\
\mb q_1 &= \tfrac{1}{2}, & \mb q_2&= \tfrac{1}{4} & \mb q_3&= \tfrac{1}{2} \\
\mb r_1 &= \tfrac{1}{4}, & \mb r_2&= \tfrac{1}{2} & \mb r_3&= \tfrac{1}{4} \\
\mb s_1 &= \tfrac{1}{4}, & \mb s_2&= 0            & \mb s_3&= \tfrac{1}{4}.
\end{align*}
For every $k \in \{1,2,3\}$ we have $\mb q_k + \mb r_k = \frac{1}{2} + \frac{1}{4}$, implying that $\sigma_k = (\mb q_k + \mb r_k - 2\mb p_k^2)^{1/2} = \frac{1}{2}$ and also that
$\gamma_1 = \gamma_3 = 0$, while $\gamma_2 = -1$. Finally, matrix $\Sigma$
is given by
\begin{equation*}
    \Sigma = \begin{bmatrix}
        1 & -1 & 0 \\ -1 & 1 & 0 \\ 0 & 0 & 1
    \end{bmatrix},
\end{equation*}
and has zero as an eigenvalue, whose eigenspace is generated by $(1,1,0)$. This means that $(X_1,X_2,X_3)$ is supported on the plane generated by $(0,0,1)$ and $(1,-1,0)$, and this plane intersects $[0,\infty)^3$ on the set $\{(0,0,z); \ge 0\}$. It follows that
for this blow up sequence the probability of intransitivity is asymptotically zero. 
\end{example}

\section{On deterministic intransitive dice}\label{sec:deterministic}

The main goal of this section is to prove Theorems~\ref{different_faces} and \ref{thm2}.
We keep using the notation introduced in Section~\ref{sec:determdice}, but in this section every die considered is deterministic, that is, the entries $D_i^{(j)}$ of the die $D^{(j)}$ in a collection $\mb D=(D^{(1)},\hdots, D^{(\ell)})$ are always prescribed deterministic numbers rather than nontrivial random variables.

As explained after Theorem~\ref{different_faces}, when investigating existence of {\it no-tie} intransitive dice, only the relative ordering of the faces of the dice matters, but not their particular values. Thus, in this section we also restrict to dice whose faces' entries are positive integer numbers, and always pairwise distinct.

\subsection{A bijection between dice and words}\label{sec:dicewords}\hfill

We look at the set of dice labels $\mc A= \{D^{(1)},\hdots, D^{(\ell)}\}$ as an alphabet, and now explain how to map dice to words. Let $\mc W(n_1,\hdots, n_\ell)$ be the set of strings (or words) with $n_1+\cdots+n_\ell$ letters in the alphabet $\mc A$, such that each letter $D^{(k)}$ appears exactly $n_k$ times. There is a natural bijection\protect\footnote{\,The idea of translating the dice as strings was inspired by a video on the YouTube channel Polylog: ``\href{https://youtu.be/-64UT8yikng}{We designed special dice using math, but there’s a catch}\textquotedblright , available at \url{https://youtu.be/-64UT8yikng}.}
 $\mc D(n_1,\hdots, n_\ell)\stackrel{\pi}{\mapsto }\mc W(n_1,\hdots, n_\ell)$: a collection of dice $\mb D$ is mapped to the word $\mb W=W_1\cdots W_n$ determined uniquely by the rule that the letter $W_i$ is equal to $D^{(k)}$ if the number $n-i+1$ appears in a face of the die $D^{(k)}$. This bijection for $\mc D(4,4,4)$ is represented in Figure \ref{fig: dado vs string}.

   \begin{figure}[!htb]
            \vspace{-2cm}
                \centering
            \begin{tikzpicture}{scale=1}
                    \dado{A}{4}{{12, 8, 4, 3}}{(-4, 3.5)}{-6/4}{{}}
                    \dado{B}{4}{{11, 7, 6, 2}}{(0, 3.5)}{-6/4}{{}}
                    \dado{C}{4}{{10, 9, 5, 1}}{(4, 3.5)}{-6/4}{{}}
                \end{tikzpicture}

                \vspace{.5cm}
                \begin{tabular}{m{0.35cm}m{0.35cm}m{0.35cm}m{0.35cm}m{0.35cm}m{0.35cm}m{0.35cm}m{0.35cm}m{0.35cm}m{0.35cm}m{0.35cm}m{0.35cm}m{0.35cm}}
                        \multicolumn{1}{c||}{{$A$}}&
                        \multicolumn{1}{c|}{12}&
                        \multicolumn{1}{c|}{}&
                        \multicolumn{1}{c|}{}&
                        \multicolumn{1}{c|}{}&
                        \multicolumn{1}{c|}{8}&
                        \multicolumn{1}{c|}{}&
                        \multicolumn{1}{c|}{}&
                        \multicolumn{1}{c|}{}&
                        \multicolumn{1}{c|}{4}&
                        \multicolumn{1}{c|}{3}&
                        \multicolumn{1}{c|}{}&
                        \multicolumn{1}{c|}{}
                        \\[1ex]
                        \multicolumn{1}{c||}{{$B$}}&
                        \multicolumn{1}{c|}{}&
                        \multicolumn{1}{c|}{11}&
                        \multicolumn{1}{c|}{}&
                        \multicolumn{1}{c|}{}&
                        \multicolumn{1}{c|}{}&
                        \multicolumn{1}{c|}{7}&
                        \multicolumn{1}{c|}{6}&
                        \multicolumn{1}{c|}{}&
                        \multicolumn{1}{c|}{}&
                        \multicolumn{1}{c|}{}&
                        \multicolumn{1}{c|}{2}&
                        \multicolumn{1}{c|}{}
                        \\[1ex]
                        \multicolumn{1}{c||}{{$C$}}&
                        \multicolumn{1}{c|}{}&
                        \multicolumn{1}{c|}{}&
                        \multicolumn{1}{c|}{10}&
                        \multicolumn{1}{c|}{9}&
                        \multicolumn{1}{c|}{}&
                        \multicolumn{1}{c|}{}&
                        \multicolumn{1}{c|}{}&
                        \multicolumn{1}{c|}{5}&
                        \multicolumn{1}{c|}{}&
                        \multicolumn{1}{c|}{}&
                        \multicolumn{1}{c|}{}&
                        \multicolumn{1}{c|}{1}\\
                        \multicolumn{1}{c||}{{$\mb W$}}&$A$&$B$&$C$&$C$&$A$&$B$&$B$&$C$&$A$&$A$&$B$&$C$
                    \end{tabular}
                    \caption{An example of a triple of dice $\mb D=(D^{(1)},D^{(2)},D^{(3)})=(A,B,C)\in \mc D(4,4,4)$ and its representation as a $12$-letter word $\pi(\mb D)=\mb W=ABCCABBCAABC\in \mc W(4,4,4)$ in the alphabet $\mc A=\{A,B,C\}$.}
                    \label{fig: dado vs string}
            \end{figure}

Recall that $\mc D_{\btt}(n_1,\hdots, n_\ell)$ is the subset of $\mc D(n_1,\hdots, n_\ell)$ that consists of intransitive words, and denote by $\mc W_\btt(n_1,\hdots, n_\ell)$ the corresponding image of $\mc D_\btt(n_1,\hdots, n_\ell)$ by the bijection $\pi$. Given the word representation $\mb W$ of a collection of dice $\mb D$, it is also possible to compare which one of two dice $D^{(i)}$ and $D^{(j)}$ in $\mb D$ is stronger: one sums how many letters $D^{(i)}$ are to the right of every letter $D^{(j)}$. The result is how many possible victories $D^{(j)}$ has over $D^{(i)}$, and if this result is larger than half the total number of combinations $n_jn_i$, then $D^{(j)} \btt D^{(i)}$. In particular, repeating this process over consecutive letters in a given word $\mb W$, it is possible to determine whether it belongs to $\mc W_\btt(n_1,\hdots, n_\ell)$.

To illustrate this process in the dice from Figure~\ref{fig: dado vs string}, introduce some auxiliary labels in the letters $\bm A$ from $\mb W$ as
\begin{equation}\label{eq:exampleW}
\mb W=ABCCABBCAABC=A_1BCCA_2BBCA_3A_4BC.
\end{equation}
There are $4$ $B$'s to the right of $A_1$, $3$ $B$'s to the right of $A_2$, and $1$ $B$ to the right of each of $A_3$ and $A_4$. Thus, the number of victories of the die $A$ over $B$ is $4+3+1+1=9$. By symmetry, the number of victories of $B$ over $A$ is $16-9=7$, and in this case $A\btt B$.

Also, to compare which of two given dice $D^{(i)}$ and $D^{(j)}$ of a collection $\mb D$ is better, it suffices to know the sub-word in $\pi(\mb D)$ obtained when we remove all letters different from $D^{(i)}$ and $D^{(j)}$. For instance, in the example just explained we could have compared the dice $A$ and $B$ by looking solely at the sub-word $ABABBAAB$ obtained when we remove the $C$'s from $\mb W$ in \eqref{eq:exampleW}.

It is convenient to introduce the quantities
\begin{equation}\label{eq:deffNijD}
N_{i,j}(\mb D)\coloneqq \sum_{\substack{ k_1,k_2 \\ D_{k_1}^{(i)}>D^{(j)}_{k_2}}} 1, \quad \text{and its induced version on }\mb W,\text{ namely} \; N_{i,j}(\mb W)\coloneqq N_{i,j}(\pi^{-1}(\mb W)).
\end{equation}
In general, for $\mb W\in \mc W(n_1,\hdots,n_\ell)$ the numbers $N_{i,j}(\mb W)$ satisfy
\begin{equation}\label{eq:NijNjideterm}
N_{i,j}(\mb W)+N_{j,i}(\mb W)=n_in_j,
\end{equation}
and the statement $\mb D^{(i)}\btt \mb D^{(j)}$ is equivalent to saying that
\begin{equation}\label{eq:ineqNijintransitivity}
N_{i,j}(\mb W)\;>\;\frac{n_in_j}{2}.
\end{equation}
Furthermore, if $\mb W$ is any sub-word of $\widetilde{\mb W}$ obtained without removing two given letters $D^{(j)}$ and $D^{(k)}$, then
\begin{equation}\label{eq:Nijsubword}
N_{i,j}(\mb W)\;=\;N_{i,j}(\widetilde{\mb W})\,,
\end{equation}
which follows from the interpretation of $N_{i,j}(\cdot)$ as the number of $D^{(i)}$'s to the left of $D^{(j)}$'s in the given letter.

\subsection{Proof of Theorem~\ref{different_faces}}\hfill

We now focus on dice with the same number of faces $n$, that is, we fix $\ell$ and look at $\mc D_\ell(n)$ and $\mc D_{\btt, \ell}(n)$ from \eqref{eq:defDellbtt}, and their corresponding images
$$
\mc W_\ell(n)\coloneqq \pi(\mc D_\ell(n)) \quad \text{and}\quad \mc W_{\btt,\ell}(n)\coloneqq \pi(\mc D_{\btt,\ell}(n))\,.
$$

The proof of Theorem~\ref{different_faces} is based on the next result.
\begin{proposition}\label{prop:intrwords}
The following properties hold.
\begin{enumerate}[(i)]
\item For $\ell \geq 3$, the sets $\mc W_{\btt,\ell}(2)$ is empty.
\item The sets $\mc W_{\btt,3}(3)$ and $\mc W_{\btt,3}(4)$ are both non-empty.
\item If the set $\mc W_{\btt,\ell}(n)$ is non-empty, then both sets $\mc W_{\btt,\ell}(n+2)$ and $\mc W_{\btt,\ell+1}(n)$ are also non-empty.
\end{enumerate}
\end{proposition}
\begin{proof}
To prove (i), let $\mb W\in \mc W_\ell(2)$ for which $D^{(1)}\btt D^{(2)}\btt \hdots \btt D^{(\ell)}$. In this case, we have that
$$
N_{j,k}(\mb W)+N_{k,j}(\mb W)=4\quad \text{for any } j\neq k,
$$
so in this case $N_{j,k}\geq 3$ whenever $D^{(j)}\btt D^{(k)}$. We learned the following: any sub-word of $\mb W$ in two different letters $D^{(j)}$ and $D^{(k)}$ for which $D^{(j)}\btt D^{(k)}$ has to be either one of the following two words
\begin{equation}\label{eq:canonicalwords22}
D^{(j)} D^{(j)} D^{(k)} D^{(k)} \quad \text{or}\quad D^{(j)} D^{(k)} D^{(j)} D^{(k)}.
\end{equation}
Thus, in $\mb W$ there is always a $D^{(1)}$ to the left of the two occurrences of $D^{(2)}$, there is always a $D^{(2)}$ to the left of the two ocurrences of $D^{(3)}$ etc. Consequently, there is always a $D^{(1)}$ to the left of the two occurrences of $D^{(k)}$, for any $k\geq 1$. Hence, the sub-word of $\mb W$ in $D^{(1)}$ and $D^{(\ell)}$ cannot be of the form \eqref{eq:canonicalwords22} with $j=\ell$ and $k=1$, so the relation $D^{(\ell)}\btt D^{(1)}$ is not verified in~$\mb W$.

For (ii), examples of words in $\mc W_{\btt,3}(3)$ and $\mc W_{\btt,3}(4)$, and their corresponding dice, are displayed in Figure~\ref{fig:33intrdice}. The proof of part (iii) is postponed to Section~\ref{sec:proofpropintrwords}.
\end{proof}
\begin{figure}[!htb]
\centering
\begin{tikzpicture}{scale=0.75}
\begin{scope}[yshift = 0cm]
\dado{A}{3}{{9, 5, 1}}{(-4, 0)}{-5/4}{{}}
\dado{B}{3}{{8, 4, 3}}{(0, 0)}{-5/4}{{}}
\dado{C}{3}{{7, 6, 2}}{(4, 0)}{-5/4}{{}}
\end{scope}
\begin{scope}[yshift = -3.5cm]
\dado{A}{4}{{12, 6, 5, 3}}{(-4, 0)}{-6/4}{{}}
\dado{B}{4}{{11, 9, 4, 2}}{(0, 0)}{-6/4}{{}}
\dado{C}{4}{{10, 8, 7, 1}}{(4, 0)}{-6/4}{{}}
\end{scope}
\end{tikzpicture}
\caption{On top: a collection of three $3$-sided dice corresponding to the word $\mb W=ABCCABBCA\in\mc W_{\btt,3}(3)$. On bottom: a collection of three $4$-sided dice corresponding to the word $\mb W=ABCBCCAABABC\in\mc W_{\btt,3}(4)$}\label{fig:33intrdice}
\end{figure}
With Proposition~\ref{prop:intrwords} at hand, we are ready to prove Theorem~\ref{different_faces}.
\begin{proof}[Proof of Theorem~\ref{different_faces}]
From the definition of $\mc D_{\btt, \ell}(n)$ it immediately follows that Theorem~\ref{different_faces} is equivalent to the following two claims:
\begin{enumerate}[(1)]
\item For any $\ell\geq 3$, the set $\mc D_{\btt, \ell}(2)$ is empty.
\item For any $n\geq 3$ and $\ell\geq 3$, the set $\mc D_{\btt, \ell}(n)$ is non-empty.
\end{enumerate}

The first one follows from the definition of the bijection $\pi$ and Proposition~\ref{prop:intrwords}--(i). In turn, claim (2) follows by applying Proposition~\ref{prop:intrwords}--(iii) inductively, having in mind that $\mc D_{\btt,3}(3)$ and $\mc D_{\btt, 4}(3)$ are both non-empty by Proposition~\ref{prop:intrwords}--(ii).
\end{proof}

\subsection{Proof of Proposition~\ref{prop:intrwords}--(iii)}\label{sec:proofpropintrwords}\hfill

To prove Proposition~\ref{prop:intrwords} we need some additional notions and lemmas. For what comes next, we recall that the numbers $N_{i,j}(\mb W)$ were introduced in~\eqref{eq:deffNijD}.

The \textbf{dual word} $\mb W^*\in \mc W(n_1,\hdots,n_\ell)$ is obtaining reversing the ordering of the letters in a given word $\mb W\in \mc W(n_1,\hdots,n_\ell)$; in the Example~\eqref{eq:exampleW} the result is
$$
\mb W^*\;=\;CBAACBBACCBA\,.
$$

\begin{lemma} \label{Lemma: nac > n*ca}
Let $\mb W\in \mc W_\ell(n)$ and $\mb W^*$ its dual word. Then
$$
N_{i,j}(\mb W)\;=\;N_{j,i}(\mb W^*)\,.
$$
\end{lemma}
\begin{proof}
The number $N_{i,j}(\mb W)$ counts the number of times the letter $D^{(i)}$ appears to the left of each $D^{(j)}$ in $\mb W$. The dual letter $\mb W^*$ is obtained from $\mb W$ by reading the letters from $\mb W$ in the backwards manner, interpretation from which the lemma follows.
\end{proof}

We say a collection of dice $\mb D\in \mc D(n_1,\hdots,n_\ell)$, or its corresponding letter $\mb W=\pi(\mb D)$, is {\bf neutral} if any given die in $\mb D$ beats any other given die in $\mb D$ the same amount of times. In terms of $N_{i,j}(\mb W)$ this is equivalent to verifying that
\begin{equation}\label{eq:Nijneutral}
N_{i,j}(\mb W)\;=\;N_{j,i}(\mb W)\,.
\end{equation}
From this relation and \eqref{eq:NijNjideterm} it follows that if $\mb W\in \mc W_\ell(n)$ is neutral, then $n$ must be even.

We also talk about the concatenation of two words $\mb W_1$ and $\mb W_2$ into a new word $\mb W=\mb W_1\mb W_2$; in the example \eqref{eq:exampleW}, for instance, we can write $\mb W=\mb W_1\mb W_2$ with $\mb W_1=ABCCA$ and $\mb W_2=BBCAABC$.

When dealing with concatenations, the involved words need not be in the same letters, neither need they have the same size. However, when $\mb W_1\in \mc W_\ell(n_1)$ and $\mb W_2\in \mc W_\ell(n_2)$, then obviously $\mb W_1\mb W_2\in \mc W_\ell(n_1+n_2)$ and the identity
\begin{equation}\label{eq:Nijconcat}
N_{i,j}(\mb W_1\mb W_2)\;=\;N_{i,j}(\mb W_1)+n_1n_2+N_{i,j}(\mb W_2)
\end{equation}
holds true.

\begin{lemma}\label{lemma: symmetric neutral}
Given any word $\mb W\in \mc W_{\ell}(n)$, the concatenation $\widetilde{\mb W}=\mb W\mb W^* \in \mc W_{\ell}(2n)$ is neutral.
\end{lemma}
\begin{proof}
Applying Lemma~\ref{Lemma: nac > n*ca} to \eqref{eq:Nijconcat}, we obtain that $N_{i,j}(\mb W)=N_{j,i}(\mb W^*)$ for any $i\neq j$. The proof is then completed using \eqref{eq:Nijneutral}.
\end{proof}

We are ready to start applying recursive arguments that preserve intransitive words, starting with adding a new letter to a word known to be transitive.

\begin{lemma}\label{lem:WellWellplus}
If $\mc W_{\btt,\ell}(n)$ is non-empty, then $\mc W_{\btt,\ell+1}(n)$ is non-empty.
\end{lemma}
\begin{proof}
From a given $\mb W\in \mc W_{\ell}(n)$, create a new word $\widetilde{\mb W}\in \mc W_{\ell+1}(n)$ obtained by replacing every occurrence of $D^{(\ell)}$ by $D^{(\ell)}D^{(\ell+1)}$. In the example $\mb W\in \mc W_{3}(4)$ from \eqref{eq:exampleW}, the new word $\widetilde{\mb W}\in \mc W_{4}(4)$ is
$$
\widetilde{\mb W}\;=\;ABCDCDABBCDAABCD.
$$
In virtue of \eqref{eq:Nijsubword}, we see that the relation $D^{(k)}\btt D^{(k+1)}$ for $k=1,\hdots, n-1$ is preserved when going from a letter $\mb W\in \mc W_{\btt,\ell}(n)$ to the corresponding letter $\widetilde {\mb W}\in \mc W_{\ell+1}(n)$. Still by construction, the relations
$$
N_{\ell,\ell+1}(\widetilde{\mb W})\;>\;N_{\ell+1,\ell}(\widetilde {\mb W}) \quad \text{and} \quad
N_{\ell,1}(\mb W)\;=\;N_{\ell+1,1}(\widetilde {\mb W})
$$
are of straightforward verification, since each letter $\mb D^{(\ell)}$ appears immediately to the right of a letter $D^{(\ell)}$ in $\widetilde {\mb W}$. When $\mb W\in \mc W_{\btt,\ell}(n)$, the inequality above shows that $D^{(\ell)}\btt D^{(\ell+1)}$ in $\widetilde{\mb W}$, and the equality above shows that the relation $D^{(\ell)}\btt D^{(1)}$ in $\mb W$ transfers to the relation $D^{(\ell+1)}\btt D^{(1)}$ in $\widetilde {\mb W}$.
\end{proof}

Adding new faces whilst preserving intransitivity is a little bit more involved, and will be based on the next two lemmas.

\begin{lemma}\label{neutral string}
Fix $k\geq 1$, and suppose that $\mb I \in\mc W_{\ell}(2k)$ is a neutral word. Let $\mb W\in \mc W_\ell(n)$. Then $\mb W\in \mc W_{\btt,\ell}(n)$ if, and only if, $\mb I\mb W\in \mc W_{\btt, \ell}(n+2k)$.
\end{lemma}

\begin{proof}
From \eqref{eq:Nijconcat} we learn that for any $i\neq j$,
$$
N_{i,j}(\mb I\mb W)\;=\;N_{i,j}(\mb I)+2kn+N_{i,j}(\mb W)\,.
$$
Using the symmetry \eqref{eq:Nijneutral} for the neutral word $\mb I$, we obtain
$$
N_{k,k+1}(\mb I\mb W)-N_{k+1,k}(\mb I\mb W)\;=\;
N_{k,k+1}({\mb W})-N_{k+1,k}({\mb W})\,,\quad k=1,\hdots, \ell-1.
$$
The result then follows from \eqref{eq:ineqNijintransitivity}.            \end{proof}

\begin{lemma}\label{lem:WnWnplusplus}
If $\mc W_\ell(n)$ is non-empty, then $\mc W_{\ell}(n+2)$ is non-empty.
\end{lemma}
\begin{proof}
By Lemma~\ref{lemma: symmetric neutral}, the word $\mb I=\mb S\mb S^*\in \mc W_{\ell}(2)$ constructed from the choice
$$
\mb S\;=\;D^{(1)}D^{(2)}\cdots D^{(\ell)}
$$
is neutral. The result now follows from Lemma~\ref{neutral string}.
\end{proof}

We finally complete the proof of Proposition~\ref{prop:intrwords}.

\begin{proof}[Proof of Proposition~\ref{prop:intrwords}--(iii)]
Proposition~\ref{prop:intrwords}--(iii) is now simply a combination of Lemmas~\ref{lem:WellWellplus} and~\ref{lem:WnWnplusplus}.
\end{proof}

\subsection{On the number of intransitive words}\hfill

The proof of Theorem~\ref{thm2} is now a consequence of some of the results already established.
\begin{proof}[Proof of Theorem~\ref{thm2}]
For given integers $n_1,n_2$, let $\mb W_j\in \mc W_{\btt,\ell}(n_j)$, $j=1,2$, so that by \eqref{eq:ineqNijintransitivity},
$$
N_{k,k+1}(\mb W_j)>\frac{(n_j)^2}{2}\,, \quad k=1,\hdots, \ell,$$
where we recall that $N_{\ell,\ell+1}=N_{\ell,1}$ by convention. Using these inequalities and \eqref{eq:Nijconcat}, it follows that $\mb W_1\mb W_2\in \mc W_\ell(n_1+n_2)$ satisfies
$$
N_{k,k+1}(\mb W_1\mb W_2)>\frac{(n_1+n_2)^2}{2}, \; k=1,\hdots, \ell.
$$
Using again \eqref{eq:ineqNijintransitivity} we conclude that $\mb W_1\mb W_2\in \mc W_{\btt,\ell}(n_1+n_2)$. Hence,
$$
|\mc W_{\btt,\ell}(n_1+n_2)|\;\geq\; |\mc W_{\btt,\ell}(n_1)|\,|\mc W_{\btt,\ell}(n_2)|\,,
$$
so the sequence $(\log |\mc W_{\btt,\ell}(n)|)_n$ is superadditive, and thus by Fekete's Lemma,
$$
\lim_{n\to \infty} \frac{ \log|\mc W_{\btt,\ell}(n)|}{n}\;=\;\sup_n \frac{ \log|\mc W_{\btt,\ell}(n)|}{n}\;=\;L(\ell)
$$
for some constant $L(\ell)>0$. Since $\mc W_{\btt,\ell}(n)=\pi(\mc D_{\btt,\ell}(n))$ and $\pi$ is a bijection, the result follows.
\end{proof}

\subsection{Some numerical aspects on the number of intransitive words}\label{sec:numerics}\hfill

Through a numerical study, we are able to estimate the number $L(3)$ from Theorem~\ref{thm2} as follows.

A simple algorithm computes $|\mc D_{\btt, 3}(n)|$ in a straightforward way: we iterate through every word in the set $\mc W_3(n)$ and check whether each word is intransitive, a task that can be accomplished in $\Theta(n)$ operations. A drawback of this approach is that the number of words that need checking grows exponentially with respect to $n$. In fact, by Stirling's approximation, we have that $|\mc D_3(n)|=\Theta(27^n/n)$, resulting in a total time complexity of $\Theta(27^n)$. We optimize this algorithm by partitioning the set $\mc W_3(n)$ into words with the same ``prefixes'' (that is, the same sequence of letters for the first $n$ positions in the word). We then avoid prefixes which are already known to not yield intransitive words, thus performing early exits while checking for intransitivity. The algorithm was implemented in C\nolinebreak\hspace{-.05em}\raisebox{.4ex}{\tiny\bf +}\nolinebreak\hspace{-.10em}\raisebox{.4ex}{\tiny\bf +} and executed on the Euler cluster maintained by the Center for Mathematical Sciences Applied to Industry (CeMEAI). Using this method, we were able to compute $|\mc D_{\btt, 3}(n)|$ for $n\leq 11$ (see Table~\ref{tab: In computado}).
            \renewcommand{\arraystretch}{1.2}
            \begin{table}[ht]
                \centering
                \vspace{0.5cm}
                \begin{tabular}{|c|c|c|c|}
                    \hline
                    $n$ & \multicolumn{1}{c|}{$|\mc D_{\btt, 3}(n)|$} & \multicolumn{1}{c|}{$|\mc D_3(n)|$} & \multicolumn{1}{c|}{\textit{$|\mc D_{\btt, 3}(n)|/|\mc D_3(n)|$}} \\ \hline
                    3                   & 15                     & 1\,680                             & $0.008\,928\ldots$              \\ \hline
                    4                   & 39                     & 34\,650                            & $0.001\,125\ldots$              \\ \hline
                    5                   & 5\,196                  & 756\,756                           & $0.006\,866\ldots$              \\ \hline
                    6                   & 32\,115                  & 17\,153\,136                         & $0.001\,872\ldots$             \\ \hline
                    7                   & 2\,093\,199                & 399\,072\,960                        & $0.005\,245\ldots$             \\ \hline
                    8                   & 19\,618\,353               & 9\,465\,511\,770                       & $0.002\,072\ldots$             \\ \hline
                    9                   & 960\,165\,789             & 227\,873\,431\,500                     & $0.004\,213\ldots$              \\ \hline
                    10                  & 11\,272\,949\,151            & 5\,550\,996\,791\,340                    & $0.002\,030\ldots$              \\ \hline
                    11                  & 479\,538\,890\,271            & 136\,526\,995\,463\,040                    & $0.003\,512\ldots$             
                    \\\hline\hline
                    50                  & $(1.47\pm0.01)\times 10^{66}$            & $2.030\,807\ldots\times 10^{69}$                    & $(7.23\pm 0.05)\times 10^{-4}$              \\ \hline
                    100                  & $(1.45\pm0.01)\times 10^{137}$            & $3.765\,234\ldots \times 10^{140}$                    & $(3.87\pm 0.03)\times 10^{-4}$              \\ \hline
                    500                  & $(2.15\pm0.03)\times 10^{708}$            & $2.649\,108\ldots \times 10^{712}$                    & $(8.1\pm 0.1)\times 10^{-5}$              \\ \hline
                    1000                  & $(2.6\pm0.1)\times 10^{1423}$            & $6.368\,665\ldots \times 10^{1427}$                    & $(4.1\pm 0.2)\times 10^{-5}$              \\ \hline
                \end{tabular}
                \vspace{1cm}
                \caption{Exact values of $|\mc D_{\btt, 3}(n)|$, $|\mc D_3(n)|$ and $|\mc D_{\btt, 3}(n)|/|\mc D_3(n)|$ for $3\le n\le 11$, along with estimated values obtained by simulation for $n=50$, $100$, $500$ and $1000$. In the values in the table, the ratios $|\mc D_{\btt, 3}(n)|/|\mc D_3(n)|$ are lower for even $n$. The decrease in ratios for even $n$ may be due to neutral strings only being possible for even $n$, leading to fewer intransitive strings proportionally.
                }\label{tab: In computado}
            \end{table}
            \renewcommand{\arraystretch}{1}

The algorithm just described yields exact values of $|\mc D_{\btt, 3}(n)|$, producing the values shown in Table~\ref{tab: In computado}. However, its performance is very slow in $n$, and even computing $|\mc D_{\btt, 3}(n)|$ for, say, $n=13$ already becomes out of reach. With this issue in mind, we also performed a stochastic simulation to estimate $|\mc D_{\btt, 3}(n)|/|\mc D_3(n)|$ for a sample of values of $n$ up to $1000$, an arbitrary cut where fluctuations in the estimation should still be controlled. We sample uniformly from $D_3(n)$, and then estimate $\Delta L_3(n)=-\log(|\mathcal{D}_{\btt, 3}(n)|/|\mathcal{D}_{3}(n))|/n$. The results are displayed in Figure \ref{fig: estimation}, showing good agreement with the values computed in Table~\ref{tab: In computado}. 
Observe how the values seem to tend to $0$. 

    Since $\lim_n \Delta L_3 (n)=3\log 3-L(3)$, one would expect that $L(3)=3\log 3$, so that the fraction $|\mc D_{\btt, 3}(n)|/|\mc D_3(n)|$ should decay sub-exponentially with $n$. In fact, in the next section we prove that $L(\ell)=\ell\log\ell$, for every $\ell\ge3$.

Compare it with \eqref{eq:subexpondecay}. All the algorithms and data presented here are publicly available in a \href{https://github.com/NonTransitiveDices/NonTransitiveDices.git}{GitHub}\footnote{https://github.com/NonTransitiveDices/NonTransitiveDices.git} repository.

\begin{figure}
\centering
    \begin{tikzpicture}[scale=.9]
	\begin{axis}[
                axis lines = left,
				ylabel={$\Delta L_3(n)$},
                ylabel style={rotate=-90},
                y label style={at={(axis description cs:-0.25,.5)},anchor=south},
                xlabel = {$n$},
    		xscale=1.0,
			    yscale=1,
                ymax = 2.5,
                ymin = 0.008,
                xmin = 1,
                xmax = 2000,
				ymajorgrids=true,
				xmajorgrids=true,
				grid style=dashed,
                ymode = log,
                xmode = log,
                legend style={at={(1.0,0.9)}}
				]
			\addplot+[
                color=red,
                mark = o,
                mark size=.5pt,
                only marks,
                error bars/.cd,
                error bar style={
                  thin,
                },
                y dir=both,
                y explicit,
                ] 
                table[
                    x index=0,
                    y index=2,
                    y error plus index=3, 
                    y error minus index=4,
                    col sep=space] {FinalMonteCarlo.dat};
                    \legend{Simulated Values}
            \addplot[color=blue,
                mark = square,
                mark size=0.5pt,
                only marks] coordinates{(3, 1.5728)(4, 1.6973)(5, 0.9962)(6, 1.0467)(7, 0.7500)(8, 0.7723)(9, 0.6077)(10, 0.6199)(11, 0.5137)};
                \addlegendentry{Exact Values}
	\end{axis}
    \end{tikzpicture}
    \caption{Plot of $\Delta L_3(n)$ for various values of $n$ on a log--log scale. The blue data points were calculated using Table~\ref{tab: In computado}, and the red data points were generated through a stochastic simulation. Note the linear trend of the plot in this case.}
    \label{fig: estimation}
\end{figure}

\section{Some generalities on the counting functions and Gaussian vectors}\label{sec:countingGaussian}

The goal here is to prove Theorem~\ref{teo:CLT}. In  this section $\mb D_m=(D^{(1)},\hdots, D^{(\ell)})$ will always denote a collection of $\ell\geq 3$ independent random dice, where each face of $D^{(i)}=D^{(i)}(m)$ has law $\mc L^{(i)}=\mc L^{(i)}_m$, and such that the sequence $\{\mb D_m\}_m$ satisfies Assumption~\ref{assumption:main}.

\subsection{Properties of the counting functions}\hfill 

Recall the counting variables $N_j$, their normalized versions $\widetilde N_j$, and $E_j$, which were defined in \eqref{eq:def_Nk}, \eqref{eq:def_tilde_Nk} and \eqref{eq:def_Ek}, respectively, and the quantities $\mb p_k,\mb q_k,\mb r_k$ and $\mb s_k$, which depend only on the collection of laws $\mc L^{(1)},\ldots, \mc L^{(\ell)}$, were introduced in \eqref{eq:defpk}--\eqref{eq:defsk}. Note that ${\bf q}_k$ and ${\bf r}_k$ are not necessarily equal because each die has its law. As said before, these quantities also depend on $m$ but we will not write down this dependence explicitly.

The next lemma establishes \eqref{eq:ENkvarNk}.

\begin{lemma}
\label{lema:mean_var_Nk}
We have
\begin{align*}
&\bb E N_k \;=\; n_k n_{k+1} \mb p_k\, , \\
&\var{N_k}
    \;=\; n_k n_{k+1}
    \bigl[n_k(\mb q_k - \mb p_k^2) + n_{k+1}(\mb r_k - \mb p_k^2)
    + \mb p_k^2 + \mb p_k -  \mb q_k -  \mb r_k
    \bigr]\, ,  \quad \text{ and} \\
& \cov{N_{k-1},N_k}=n_{k-1}n_kn_{k+1}\left(\mb s_k-\mb p_{k-1}\mb p_k\right).
\end{align*}
Consequently, under Assumption~\ref{assumption:main} we have that as $m\to \infty$,
\begin{align*}
&\bb E N_k  \;=\; f_k(\infty) f_{k+1}(\infty) \mb p_k(\infty)m^2 +o(m^2)\,,\\
 &\var{N_k} \;=\; \sigma_k(\infty)^2 m^3 +
    O(m^2)\, ,\\
 & \corr{N_{k-1},N_{k}}\;=\; \gamma_k(\infty) +O(1/m).
\end{align*}
\end{lemma}

\begin{proof}
The calculation of $\bb E N_k$ is immediate from the definition. For the
variance, we begin noticing that
\begin{align*}
\esp{N_k^2}
    &\;=\; \ \sum_{i_1=1}^{n_k}\sum_{j_1=1}^{n_{k+1}}
    \sum_{i_2=1}^{n_k}\sum_{j_2=1}^{n_{k+1}}
    \bb P\left( \{ D^{(k)}_{i_1} > D^{(k+1)}_{j_1} \}\cap\{
    D^{(k)}_{i_2} > D^{(k+1)}_{j_2}\}\right).
\end{align*}
The probability of such intersections is always in
$\{\mb p_k^2, \mb p_k, \mb r_k, \mb q_k\}$, depending on whether
indices $i_1$ and $i_2$ or $j_1$ and $j_2$ coincide.
Decomposing into all possibilities, we have
\begin{align*}
\esp{N_k^2}
    &=\sum_{\substack{i_2 \neq i_1\\~j_2\neq j_1}} \mb p_k^2
    +\sum_{\substack{i_2=i_1\\~j_2\neq j_1}} \mb r_k
    +\sum_{\substack{j_2=j_1\\~i_2 \neq i_1}} \mb q_k
    +\sum_{\substack{i_2 = i_1\\~j_2= j_1}} \mb p_k
    \\
    &=n_k(n_k-1)n_{k+1}(n_{k+1}-1) \mb p_k^2
    +n_k n_{k+1}(n_{k+1}-1) \mb r_k\\
    &\qquad +n_k(n_k-1)n_{k+1} \mb q_k
    +n_k n_{k+1} \mb p_k.
\end{align*}
Hence, the variance of $N_k$ is given by
\begin{align*}
\var{N_k}
    &\;=\; \esp{N_k^2} - \esp{N_k}^2
    \;=\; \esp{N_k^2} - (n_k n_{k+1} \mb p_k)^2 \\
    &\;=\; n_k n_{k+1}
    \bigl[n_k(\mb q_k - \mb p_k^2) + n_{k+1}(\mb r_k - \mb p_k^2)
    + \mb p_k^2 + \mb p_k -  \mb q_k -  \mb r_k
    \bigr].
\end{align*}
The calculation of $\cov{N_{k-1},N_k}$ is similar, and the asymptotic expressions follow immediately.
\end{proof}

A case of particular interest is when all the underlying laws are the same, and given by a continuous distribution.

\begin{lemma}\label{lem:prqscontdistr}
    Assume that $\mb D=(D^{(1)},\hdots, D^{(\ell)})$ is a collection of dice for which $\mc L^{(1)}=\cdots = \mc L^{(\ell)}$, with $\mc L^{(1)}$ being a law with no mass points. Then
    $$
    \mb p_k=\frac{1}{2},\quad \mb q_k=\frac{1}{3}=\mb r_k \quad \text{and}\quad \mb s_k=\frac{1}{6}
    $$
    for every $k$. In particular, if $n_1=\cdots = n_\ell\eqqcolon n$, then
    $$
    \esp{N_k}=\frac{n^2}{2},\quad  \var{N_k}=\frac{n^3}{6}\left( 1+\frac{1}{2n} \right), \quad \cov{N_{k-1},N_k}=-\frac{n^3}{12} \quad \text{and}\quad \gamma_k=-\frac{1}{2}.
    $$
\end{lemma}
\begin{proof}
    Let $X_1,X_2,X_3$ be three independent random variables with distribution $\mc L^{(1)}$. Because $\mc L^{(1)}$ has no mass points, we have
    \begin{equation}\label{eq:nomassprobs}
    \bb P (X_1=X_2)=\bb P (X_2=X_3)=\bb P (X_3=X_1)=0,
    \end{equation}
    which we use extensively in what follows.
    
    Because $X$ and $Y$ have the same distribution, we obtain
    $$
    \mb p_k=\bb P(X_1>X_2)=\bb P(X_2>X_1)=\frac{1}{2}\left(\bb P(X_1>X_2)+\bb P(X_2\geq X_1)\right)=\frac{1}{2}.
    $$
    In a similar spirit, 
    $$
    \mb q_k=\bb P(X_2>X_1,X_3>X_1)=\bb P(X_1>X_2,X_3>X_2)=\bb P(X_1>X_3,X_2>X_3).
    $$
    Summing up and using \eqref{eq:nomassprobs} we obtain the claimed value $\mb q_k=1/3$. The values of $\mb r_k$ and $\mb s_k$ are computed in a similar manner, and the claimed values of the mean, variance, covariance and $\gamma_k$ are then a consequence of Lemma~\ref{lema:mean_var_Nk} and \eqref{eq:def_sigma_k}--\eqref{eq:def_gamma_k}.
\end{proof}

The quantities $E_k$ from \eqref{eq:def_Ek} admit a result similar to Lemma~\ref{lema:mean_var_Nk}. In order to state and prove it, we need to introduce certain quantities analogous to $\mb p_k,\mb q_r, \mb r_k, \mb s_k$. We introduce
\begin{equation}\label{eq:defpkequal}
{\bf p}_k^= \;\coloneqq\; \bb P\left(D^{(k)}_1= D^{(k+1)}_1\right)\;=\;\bb E\left(\ind_{D_1^{(k)}=D_1^{(k+1)}}\right)
\end{equation}
which is the that probability a given face of the $k$-th die coincide with a given face of the  $(k+1)$-th die;
\begin{equation}\label{eq:defqkequal}
{\bf q}_k^= \;\coloneqq\; \bb P\left(D^{(k)}_1= D^{(k+1)}_1, D^{(k)}_2= D^{(k+1)}_1\right)
\end{equation}
which is the probability that two given faces of the $k$-th die coincide with a given face of the  $(k+1)$-th die;
\begin{equation}\label{eq:defrkequal}
{\bf r}_k^= \;\coloneqq\; \bb P\left(D^{(k)}_1= D^{(k+1)}_1, D^{(k)}_1= D^{(k+1)}_2\right)
\end{equation}
which is the probability that two given faces of the $(k+1)$-th die coincide with a given face of the  $k$-th die; and
\begin{equation}\label{eq:defskequal}
{\bf s}_k^= \;\coloneqq\; \bb P\left(D^{(k-1)}_1 = D^{(k)}_1 = D^{(k+1)}_1\right),
\end{equation}
which is the probability that three given faces, one from each of the dice $D^{(k-1)}$, $D^{(k)}$ and $D^{(k+1)}$, coincide.

The next result is the analogue of Lemma~\ref{lema:mean_var_Nk} for the variables $E_k$'s.

\begin{lemma}\label{lem:meanvarEk}
The random variables $E_k$, $k=1,\hdots,\ell$, satisfy
\begin{align}
\bb E \left(E_k\right) &= n_k n_{k+1} \mb p_k^=  \label{eq:Ek_mean}\\
\var{E_k}
    &= n_k n_{k+1}
    \bigl[n_k(\mb q_k^= - (\mb p_k^=)^2) + n_{k+1}(\mb r_k^= - (\mb p_k^=)^2) 
    + (\mb p_k^{=})^2 +\mb p_k^= - \mb q_k^= - \mb r_k^=
    \bigr].\label{eq:Ek_variance}
\end{align}
In particular, the estimates \eqref{eq:variancetiesgeneral} hold true.
\end{lemma}
\begin{proof}
The proof of \eqref{eq:Ek_mean}--\eqref{eq:Ek_variance} is done following the same steps used in the proof of Lemma~\ref{lema:mean_var_Nk}, we skip the details. The estimates \eqref{eq:variancetiesgeneral} then follow, having in mind \eqref{deff:fkm} and the fact that each $\mb p_k^=, \mb q_k^=, \mb r_k^=, \mb s_k^=$ is a probability, and thus bounded as functions of $m$.
\end{proof}

Using the previous result, we are able to compare \eqref{eq:asymptoticnotie} with \eqref{eq:variancetiesgeneral}.

\begin{lemma}\label{lem:notiemeanvariance}
If condition \eqref{eq:asymptoticnotie} holds, then for every $k=1,\hdots,\ell$, it is valid
$$
\bb E(E_k)= o(m^2)  \quad \text{and}\quad \var{E_k}= o(m^3) \quad \text{as}\quad m\to \infty.
$$
\end{lemma}
\begin{proof}
Condition \eqref{eq:asymptoticnotie} is the same as saying that $\bm p_k^=\to 0$ for every $k$. Thus, the claim on $\bb E(E_k)$ is immediate from \eqref{eq:Ek_mean}.
From \eqref{eq:Ek_variance} and the fact that $\mb p_k^=, \mb q_k^=$ and $\mb r_k^=$ remain bounded as $m\to\infty$, we see that
$$
\var{E_k}\;=\;
m^3 f_k(\infty)f_{k+1}(\infty)
\bigl[f_k(\infty)(\mb q_k^=-(\mb p_k^=)^2) + f_{k+1}(\infty)(\mb r_k^=-(\mb p_k^=)^2)\bigr]+O(m^2)\,.
$$
A comparison of \eqref{eq:defpkequal} and \eqref{eq:defqkequal} shows that $0\leq \max\{\mb q_k^=, \mb r_k^= \}\leq \mb p_k^=$, so that \eqref{eq:asymptoticnotie} implies also that $\mb q_k^=, \mb r_k^= \to 0$ and the claim on $\var{E_k}$ follows.
\end{proof}

Next, we turn our attention to the structure of the covariance matrix \eqref{eq:covariancematrixCLT}. It turns out that in the case of particular interest to our problem, the coefficients $\gamma_j(\infty)$ have a special structure that we now compute.

\begin{proposition}\label{prop:coeffgammak}
Let $\{\mb D_m\}_m$ be a sequence satisfying \eqref{eq:assumptionpqrssym}. Then the coefficients $(\gamma_k(\infty))$ and $(f_k(\infty))$ from Assumption~\ref{assumption:main} are related by
$$
\gamma_k(\infty)\;=\;-\frac{f_{k-1}(\infty)f_k(\infty)f_{k+1}(\infty)}{\sqrt{f_{k-1}(\infty)f_k(\infty)(f_{k-1}(\infty)+f_{k}(\infty))}\sqrt{f_{k}(\infty)f_{k+1}(\infty)(f_{k}(\infty)+f_{k+1}(\infty))}}
$$
for $k=1,\hdots,\ell$.
\end{proposition}
\begin{proof}
These are simple calculations using \eqref{eq:def_gamma_k}.
\end{proof}

\subsection{Gaussian vectors associated to the structured covariance matrix}\hfill 


Under the condition \eqref{eq:assumptionpqrssym},
Proposition~\ref{prop:coeffgammak} ensures that the nontrivial entries
$\gamma_k(\infty)$ of the covariance matrix \eqref{eq:covariancematrixCLT} have
a particular structure, which ultimately yields that the probability in the
right-hand side of \eqref{eq:intransitivelimit_ineq} vanishes, and proving this last claim is the main goal of this subsection.

To avoid cumbersome notation, for the calculations that come next we denote
$$
\f_k \;\coloneqq\; f_k(\infty),\quad k=1,\hdots, \ell, \quad \f_{\ell+1}\coloneqq f_1(\infty)\,.
$$

Thanks to Proposition~\ref{prop:coeffgammak},
\begin{equation}\label{eq:gammakfrakf}
\gamma_k(\infty)\;=\;-\frac{\f_{k-1}\f_k\mf f_{k+1}}{\sqrt{\f_{k-1}\mf f_k(\f_{k-1}+\f_{k})}\sqrt{\f_{k}\f_{k+1}(\f_{k}+\f_{k+1})}}\,,
\end{equation}
and now we study the covariance matrix $\Sigma$ from \eqref{eq:covariancematrixCLT} with coefficients given as in \eqref{eq:gammakfrakf}. We start by collecting some properties of these $\gamma_k(\infty)$'s.

\begin{proposition}\label{prop:propertiescoeffgammak}
The coefficients $\gamma_k=\gamma_k(\infty)$, $k=1,\hdots,\ell$, in \eqref{eq:gammakfrakf} satisfy the following properties.
\begin{enumerate}[(i)]
\item $\displaystyle\gamma_k^2 = \frac{\f_{k-1}}{\f_{k-1}+\f_k} \cdot
    \frac{\f_{k+1}}{\f_{k}+\f_{k+1}}$.
\item $\gamma_k\in (-1,0)$ for every $k$.
\item As functions of the $\f_j$'s, the coefficients $\gamma_k=\gamma_k(\f_1, \ldots, \f_\ell)$ are scale-invariant: for every $k \in
    [\ell]$ and $r > 0$ we have
\begin{equation*}
\label{eq:gamma_k_scale_inv}
\gamma_k(r\f_1, \ldots, r\f_\ell)
    = \gamma_k(\f_1, \ldots, \f_\ell).
\end{equation*}
\item $\prod_{k} \gamma_k = (-1)^\ell \prod_k \frac{\f_k}{\f_k + \f_{k+1}}$.
\item $|\prod_k \gamma_k| \le 2^{-\ell}$, with equality being valid if, and only if, $\f_1=\cdots=\f_\ell$.
\end{enumerate}
\end{proposition}

\begin{proof}
Items \textit{(i)}, \textit{(iii)} and \textit{(iv)} are immediate
from~\eqref{eq:gammakfrakf}. It is obvious that $\gamma_k<0$, so to prove \textit{(ii)} it suffices to show that $\gamma_k^2<1$ which, in turn, by {\it (i)} is equivalent to the inequality
\begin{equation*}
\f_{k-1}\f_{k+1}<(\f_{k-1}+\f_k)(\f_{k}+\f_{k+1}),\quad \text{that is}, \quad 0<\f_{k-1}\f_k+\f_k^2+\f_k\f_{k+1}.
\end{equation*}
Since $\f_k>0$ for every $k$, part \textit{(ii)} follows.
Finally, for item \textit{(v)} we apply the  inequality between arithmetic and geometric means to obtain
\begin{equation*}
\frac{f_k + f_{k+1}}{2} \ge \sqrt{f_k f_{k+1}},\
    \text{for every }k \in [\ell].
\end{equation*}
Multiplying all inequalities above, the result follows using
item \textit{(iv)}.
\end{proof}

With the aforementioned properties of $\gamma_k=\gamma_k(\infty)$ from \eqref{eq:gammakfrakf} at hand, we now need to collect some important information on the associated covariance matrix $\Sigma$ from \eqref{eq:covariancematrixCLT}. From a linear algebra perspective, this is an example of a periodic Jacobi matrix (see for instance \cite{vanMoerbekeJac, molinari2008determinants, FergusonWarren}). However, we could not explore these interpretations for the results needed later. Instead, in our case, we use the additional structure \eqref{eq:gammakfrakf} in a fundamental way to obtain the next results.

\begin{lemma}
\label{lema:det_Sigma_expression}
Let $\Sigma$ be as in \eqref{eq:covariancematrixCLT} with coefficients $\gamma_k=\gamma_k(\infty)$ as in \eqref{eq:gammakfrakf}. Then
\begin{equation}
\label{eq:det_Sigma_expression}
\det \Sigma
    \;=\; 1 + 2 (-1)^{\ell-1} \gamma_{1} \ldots \gamma_{\ell} + \sum_{m=1}^{\ell}
    (-1)^{m} \sum_{\mathclap{\substack{ i_1 < i_2 < \cdots < i_m:\\
    i_j\text{ non-consecutive}}}}
        \ \gamma_{i_1}^2 \ldots \gamma_{i_m}^2\,,
\end{equation}
where in the sum above, the indices $1$ and $\ell$ are considered as consecutive.
\end{lemma}

\begin{proof}
We start from the expression of $\det \Sigma$ as a sum over all
permutations $\sigma \in S_\ell$, the symmetric group of degree $\ell$,
\begin{equation*}
\det \Sigma
    \;=\; \sum_{\sigma \in S_\ell} \sgn \sigma \cdot
        \prod_{i \in [\ell]} \Sigma_{i\sigma(i)}\,.
\end{equation*}
To prove \eqref{eq:det_Sigma_expression} we will now show that many terms do not contribute to the sum.

From the explicit form of $\Sigma$, we can see that the permutations $\sigma$ such that $\sigma(i) \notin \{i-1,i,i+1\}$ for some $i$ have $\Sigma_{i\sigma(i)} = 0$,
recalling that we consider indices modulo $\ell$.

For the remaining terms, consider the cyclic decomposition of $\sigma = \tau_1 \cdots \tau_m$ where $\tau_j$ are disjoint cycles. Using the disjointness, we can compute the
product of $\Sigma_{i\sigma(i)}$ by evaluating it for each cycle.

The contribution of cycles of order $1$ is always $1$, since when $\tau=(i)$ we have that $\Sigma_{i\tau(i)} = \Sigma_{ii} = 1$.

The contribution of a cycle of order $2$, say $\tau=(i\ j)$, is non-zero if, and only if, $i$ and $j$ are consecutive.  If $i=j-1$ then $\Sigma_{j-1,j} \Sigma_{j,j-1} = \gamma_j^2$.

By a similar reasoning, for a cycle $\tau = (i_1\ i_2\ \cdots\ i_o)$ of order
$o \ge 3$ to have a non-zero contribution one must have $o=\ell$. Indeed, we
have $i_2 \in \{i_1 - 1, i_1 + 1\}$. If $i_2 = i_1 + 1$, since all indices must
be consecutive we have $i_j = i_{j-1}+1$ for every $j$. After $i_o$ the cycle
returns to $i_1$, implying that it has length $\ell$. The case $i_2=i_1-1$ is
analogous.

Hence, apart from cycles of order $1$ and cycles of the form $(i\ i+1)$, the only
cycles whose product is non-zero are $(1\ 2\ \cdots\ \ell)$ and
$(1\ \ell\ \ell-1\ \cdots\ 2)$. Both have the same product:
\begin{equation*}
\prod_{i \in [\ell]} \Sigma_{i, i+1} \;=\; \prod_{i \in [\ell]} \gamma_i\,.
\end{equation*}
Finally, if $\sigma$ is a permutation different from the identity and the two
cycles of order $\ell$, in order for its product to be non-zero one must have a
cyclic decomposition $\sigma = \tau_1 \ldots \tau_t$ with every $\tau_i$ being
a cycle of order $1$ or $2$. As cycles of order $1$ contribute with $1$ to the
product, we can focus on the cycles of order $2$. Suppose there are $m$ cycles of
order $2$ and reorder if necessary so that they are given by $\tau_1, \ldots, \tau_m$
with $\tau_j = (i_j-1\ i_j)$ and $i_1 < i_2 < \cdots < i_m$. The formula
in~\eqref{eq:det_Sigma_expression} follows, since the $i_j$ are non-consecutive
by construction.
\end{proof}

Using Lemma~\ref{lema:det_Sigma_expression} we are able to verify that $\det
\Sigma$ is always zero.

\begin{lemma}
\label{lem:det_Sigma_is_zero}
Let $\Sigma$ be as in \eqref{eq:covariancematrixCLT} with coefficients $\gamma_k=\gamma_k(\infty)$ as in \eqref{eq:gammakfrakf}. Then $\det \Sigma = 0$.
\end{lemma}
\begin{proof}
We have to prove that the right-hand side of~\eqref{eq:det_Sigma_expression} is
zero. We will replace the expressions for $\gamma_k=\gamma_k(\infty)$ given by Proposition~\ref{prop:propertiescoeffgammak}--{\it (i)}, {\it (iv)} and verify that the right-hand side of \eqref{eq:det_Sigma_expression} vanishes. In order to make the computation more streamlined, it is convenient to reinterpret it as an estimate of probabilities, as we describe below.

Let us define $a_k \coloneqq \frac{\f_{k-1}}{\f_{k-1}+f_k} \in (0,1)$, which satisfies
$1 - a_{k+1} = \frac{\f_{k+1}}{\f_{k}+\f_{k+1}}$. Consider a collection
$(U_j: j \in [\ell])$ of i.i.d. uniform random variables in $(0,1)$ and set
\begin{equation}
\label{eq:def_event_Ak}
    A_k \;\coloneqq\; \{U_k \le a_k\}\,.
\end{equation}
The collection $(A_k: k \in [\ell])$ consists of mutually independent events such that $\bb P(A_k) = a_k$.  Defining $B_k \coloneqq A_k \cap A_{k+1}^{c}$, it holds that
\begin{equation}
\label{eq:prob_Bk}
\bb P(B_k)\; = \; a_k (1 - a_{k+1}) \;=\; \gamma_k^2\,.
\end{equation}
Let us compute the probability of the event $\cup_k B_k$ in two different ways.
By the inclusion-exclusion principle, we have
\begin{align}
\bb P(\cup B_k)
    &\;=\; \sum_{m=1}^{\ell} (-1)^{m-1} \sum_{\mathclap{i_1 < i_2 < \cdots < i_m}}
        \ \bb P(B_{i_1} \cap \cdots \cap B_{i_m}) \;=\; \sum_{m=1}^{\ell} (-1)^{m-1} \sum_{\mathclap{\substack{ i_1 < i_2 < \cdots
    < i_m \\ i_j\text{ non-consecutive} }}}
        \ \bb P(B_{i_1} \cap \ldots \cap B_{i_m}) \nonumber\\
\label{eq:prob_cup_Bk_1}
    &\;=\; \sum_{m=1}^{\ell} (-1)^{m-1} \sum_{\mathclap{\substack{ i_1 < i_2 < \ldots
    < i_m \\ \text{$i_j$ non-consecutive} }}}
        \ \gamma^2_{i_1} \ldots \gamma^2_{i_m}\,,
\end{align}
where the second equality is due to
$B_k \cap B_{k+1} = (A_k \cap A^{c}_{k+1}) \cap (A_{k+1} \cap A^{c}_{k+2}) =\varnothing$,
and the last equality follows from~\eqref{eq:prob_Bk} and independence.
On the other hand, from the identity
\begin{equation*}
    \cap B^{c}_k \;=\; (\cap_{k\in [\ell]} A_k) \cup (\cap_{k\in [\ell]} A^c_k)
\end{equation*}
we compute
\begin{align}
\bb P(\cup B_k)
    &\;=\; 1 - \bb P(\cap B^c_k)
    \;=\; 1 - \bb P(\cap A^c_k) -  \bb P(\cap A_k)\nonumber\\
    &\;=\; 1 - a_1 \cdots a_\ell - (1 - a_1) \cdots (1 - a_\ell) \nonumber\\
    &\;=\; 1 - 2 \prod \frac{\f_k}{\f_k + \f_{k+1}}     = 1 - 2 (-1)^\ell \prod \gamma_k\,.\label{eq:prob_cup_Bk_2}
\end{align}
Equating~\eqref{eq:prob_cup_Bk_1} and \eqref{eq:prob_cup_Bk_2} the result
follows.
\end{proof}

Lemma~\ref{lem:det_Sigma_is_zero} ensures that zero is an eigenvalue of $\Sigma$, and we now collect some info about the associated eigenspace.

\begin{proposition}
\label{prop:h2_eigenspace_zero}
Let $\Sigma$ be as in \eqref{eq:covariancematrixCLT} with coefficients $\gamma_k=\gamma_k(\infty)$ as in \eqref{eq:gammakfrakf}. Then $0$ is an eigenvalue of $\Sigma$, its eigenspace has dimension 1 and is generated by a vector $x \in (0,\infty)^{\ell}$.
\end{proposition}
\begin{proof}
Let $x=(x_1, \dots, x_\ell)$ be a non-zero vector satisfying $\Sigma x = 0$.
Then, for every $k \in [\ell]$ we have
\begin{equation}
\label{eq:eigen_0_Lk} \tag{$L_k$}
\gamma_{k} x_{k-1} + x_{k}+\gamma_{k+1} x_{k+1} \;=\;0\,.
\end{equation}

It is possible to solve the system of equations above explicitly, but the formulas obtained this way are cumbersome. Instead, we show that if some coordinate $x_j$ is positive then $x_{j+1}$ is positive as well. Since we can always choose the sign of one entry of $x$, by the cyclic symmetry of the problem we then conclude that there is $x \in (0,\infty)^{\ell}$ with $\Sigma x = 0$, as wanted.
Therefore, assume without loss of generality that $x_{\ell-1}\geq 0$. From
$(L_1)-\gamma_1(L_\ell)$, we obtain
\begin{equation*}
 -\gamma_1\gamma_\ell x_{\ell-1} +(1-\gamma_1^2)x_1+\gamma_2 x_2\;=\;0\,.
\end{equation*}
Defining $P_0 := 1$ and $P_1 := 1 - \gamma_1^2$, the equation above becomes
\begin{equation}
 -\gamma_1\gamma_\ell x_{\ell-1} +P_1 x_1+\gamma_2 P_0 x_2\;=\;0 \tag{$L'_1$}\,.
\end{equation}
Equation $(L'_1)$ relates $x_{\ell-1}$ to $x_1$ and $x_2$. By successive applications
of the same reasoning we can relate $x_{\ell-1}$ to $x_k$ and $x_{k+1}$ for any
$k$. Indeed, suppose that it holds
\begin{equation}
\tag{$L'_k$}
(-1)^k\gamma_1\dots\gamma_k\gamma_\ell x_{\ell-1}+P_k x_k+\gamma_{k+1}
    P_{k-1}x_{k+1}=0,
\end{equation}
for some already defined $P_{k-1}$ and $P_k$.
Then, from $P_k(L_{k+1})-\gamma_{k+1}(L'_k)$, we obtain
\begin{align*}
    &\bigl(P_k \gamma_{k+1} x_{k}
    + P_k x_{k+1}+P_k \gamma_{k+2} x_{k+2}\bigr) \\
    &\qquad - \bigl(\gamma_{k+1}(-1)^k\gamma_1\dots\gamma_{k}\gamma_\ell x_{\ell-1}
        + \gamma_{k+1} P_k x_k+\gamma_{k+1}^2 P_{k-1}x_{k+1}\bigr) \\
    &= (-1)^{k+1}\gamma_1\dots\gamma_{k+1}\gamma_\ell x_{\ell-1}+
    (P_k -\gamma_{k+1}^2 P_{k-1})x_{k+1}+\gamma_{k+2} P_{k}x_{k+2}
\;=\; 0\,.
\end{align*}
Defining $P_{k+1} \coloneqq P_k - \gamma_{k+1}^2 P_{k-1}$ for $k\leq \ell-2$, we conclude that $(L'_{k+1})$
also holds. Since we know that $(L'_1)$ holds, it follows by induction that
$(L'_{\ell-2})$ holds, implying that
\begin{align*}
0 \;=\; (-1)^{\ell-1}\gamma_1\dots\gamma_\ell x_{\ell-1}+
    P_{\ell-1} x_{\ell-1}+\gamma_{\ell} P_{\ell-2}x_{\ell}\,,
\end{align*}
that is,
\begin{align*}
\gamma_{\ell} P_{\ell-2}x_{\ell}
\;=\; \bigl((-1)^{\ell}\gamma_1\dots\gamma_\ell - P_{\ell-1}\bigr) x_{\ell-1}\,.
\end{align*}
To finish the proof, we need to control the sign of the coefficients appearing
above. Once again, the strategy of expressing relevant quantities using the
independent events $A_k$ plays a role.

\begin{lemma}
\label{lema:seq_Pk}
Define 
$$P_0 \coloneqq 1,\quad  P_1\coloneqq 1 - \gamma^2_1,\quad  P_k \coloneqq P_{k-1} - \gamma_k^2
P_{k-2}, \; 2\leq k\leq \ell-1,$$
and
$$P_\ell\coloneqq 2(-1)^\ell \gamma_1\cdots \gamma_\ell.
$$ 
Then
\begin{equation}
\label{eq:seq_Pk}
1    \;=\; P_0 \;>\; P_1 \;>\; P_2 \;>\; \cdots \;>\; P_{\ell-1} > P_{\ell}
    \;=\; 2(-1)^{\ell} \gamma_1 \cdots \gamma_\ell
    \;>\;0\,.
\end{equation}
\end{lemma}

\begin{proof}
Recall $A_k$ as defined in~\eqref{eq:def_event_Ak} and $B_k = A_k \cap
A_{k+1}^c$. We simply notice that the sequence $P_k$ in the statement can be
alternatively described by the equation
\begin{equation}\label{eq:seq_Pk_as_union}
    P_k \;=\; 1 - \bb P \Bigl(\bigcup_{j=1}^k B_j\Bigr),\quad k\leq \ell-1.
\end{equation}
Indeed, it is straightforward to check~\eqref{eq:seq_Pk_as_union} for $k=0,1$.
Now, suppose that~\eqref{eq:seq_Pk_as_union} holds for $k-1$. Then
\begin{align*}
1 - \bb P\Bigl( \bigcup_{j=1}^k B_j\Bigr)
    \;=\; 1 - \bb P\Bigl( \bigcup_{j=1}^{k-1} B_j\Bigr) - \bb P (B_k)
    + \bb P\Bigl( B_k \cap \bigcup_{j=1}^{k-1} B_j\Bigr).
\end{align*}
Since $B_k \cap B_{k-1} = \varnothing$, the intersection above is given by
\begin{equation*}
\bb P\Bigl( B_k \cap \bigcup_{j=1}^{k-1} B_j\Bigr)
\;=\; \bb P\Bigl( B_k \cap \bigcup_{j=1}^{k-2} B_j\Bigr)
    \;=\; \bb P(B_k) \bb P\Bigl(\bigcup_{j=1}^{k-2} B_j\Bigr)
    \;=\; \gamma_k^2 (1 - P_{k-2})\,,
\end{equation*}
where in the second identity we used that $B_k=A_k\cap A_{k+1}^c$ and $\cup_{j\leq k-1} B_j=\cup_{j\leq k-2}(A_j\cap A_{j+1}^c)$ are mutually independent, because the $A_j$'s are.
Putting together the equations above, we obtain
\begin{equation*}
1 - \bb P\Bigl( \bigcup_{j=1}^kB_j\Bigr)
    \;=\; P_{k-1} - \gamma_k^2 + \gamma_k^2 (1-P_{k-2})
    \;=\; P_{k-1} - \gamma_k^2 P_{k-2}
    \;=\; P_k\,,
\end{equation*}
completing the induction step.
The inequalities in~\eqref{eq:seq_Pk} are now evident from the fact that
the sequence of events $\cup_{j=1}^k B_j$ is increasing in $k$ and that
we already know $P_{\ell} = \bb P(\cap B_j^c) = 2 (-1)^{\ell} \gamma_1 \dots
    \gamma_{\ell} > 0$, see~\eqref{eq:prob_cup_Bk_2}.
\end{proof}

With Lemma~\ref{lema:seq_Pk}, we can finish the proof by noticing that
\begin{equation*}
\gamma_\ell P_{\ell-2} \;<\;0
\quad \text{and} \quad
(-1)^{\ell}\gamma_1\dots\gamma_\ell - P_{\ell-1}
\;=\; \frac{1}{2}P_{\ell} - P_{\ell-1} \;<\; 0\,,
\end{equation*}
which imply that $x_{\ell-1}$ and $x_{\ell}$ have the same sign. The reasoning
above actually shows that any eigenvector of zero with some positive entry is in $(0,\infty)^\ell$. Finally, we argue that the eigenspace of zero
has dimension 1. The Spectral Theorem ensures that $\Sigma$ has
an orthonormal basis of eigenvectors. Now, suppose that $v_1, v_2$ are two
orthogonal eigenvectors of zero. Replacing $v_j$ by $-v_j$ if needed, we can
assume $v_j \in (0, \infty)^\ell$ for $j=1,2$. But then their inner product is
positive, leading to a contradiction.
\end{proof}

The final result of this section is a consequence of the previous proposition, and it is the essential outcome of this section which will be used later.

\begin{theorem}\label{thm:probabilitieszero}
For $\ell\geq 3$, suppose that $X=(X_1,\hdots, X_\ell)$ is a centered Gaussian vector with covariance matrix $\Sigma$ as in \eqref{eq:covariancematrixCLT}, whose coefficients $\gamma_k=\gamma_k(\infty)$ are of the form \eqref{eq:gammakfrakf}. Then
$\bb P(X_j\geq 0,\;  j=1,\hdots, \ell)=0$.
\end{theorem}
\begin{proof}
Recall that the support of a Gaussian vector $Z$ is given by $\bb E(Z)+\mathrm{Ker}(\cov{Z})^\perp$. Thus, in our case the support of $X$ is $\mathrm{Ker}(\Sigma)^\perp$, and by Proposition~\ref{prop:h2_eigenspace_zero}, $\mathrm{Ker}(\Sigma)$ is spanned by a vector $v=(v_1,\hdots,v_\ell)$ with $v_j>0$ for every $j$. If $y=(y_1,\hdots,y_\ell)$ is such that $y_j>0$ for every $j$, then we must have $\langle y,v\rangle >0$, so $y\notin \mathrm{Ker}(\Sigma)^\perp$. Thus,
    $$
        \mathrm{supp}(X)\cap \{y\in \bb R^d: y_j>0, j=1,\hdots, \ell\}\;=\;\varnothing\,,
    $$
and the result follows.
\end{proof}

\section{Proofs of Theorems~\ref{thm:intransitive} and \ref{thm:nointransitive}}\label{s:thms5_and_6}

Theorem~\ref{teo:CLT} will be proved in the next section. In this section we assume its validity in order to prove Theorems~\ref{thm:intransitive} and \ref{thm:nointransitive}.

\begin{proof}[Proof of Theorem~\ref{thm:intransitive}]
We will prove the following two properties
\begin{enumerate}[(1)]
\item 
If condition \eqref{eq:asymptoticnotie} holds true, and in addition there is a function $r(m)$ with
$\lim_{m\to\infty}r(m)=+\infty$, for which
\begin{equation}\label{eq:condttslowgrowthpk_limsup}
\frac{1}{2}-\mb p_k-\frac{1}{2}\bb P(D_1^{(k)}=D_1^{(k+1)})
\;\geq\; -\frac{1}{m^{1/2}r(m)}\,,\quad k=1,\hdots, \ell,
\end{equation}
then
\begin{equation}\label{eq:intransitivelimit_limsup}
\limsup_{m\to \infty}\,\bb P\left({D^{(1)} \btt \cdots \btt D^{(\ell)} \btt D^{(1)}}\right)\;\leq\; \bb P\left(X_j\geq 0, j=1,\hdots, \ell \right).
\end{equation}

\item
If condition \eqref{eq:asymptoticnotie} holds true, and in addition there is a function $r(m)$ with
$\lim_{m\to\infty}r(m)=+\infty$, for which
\begin{equation}\label{eq:condttslowgrowthpk_liminf}
\frac{1}{2}-\mb p_k-\frac{1}{2}\bb P(D_1^{(k)}=D_1^{(k+1)})
\;\leq\; \frac{1}{m^{1/2}r(m)}\,,\quad k=1,\hdots, \ell,
\end{equation}
then
\begin{equation}\label{eq:intransitivelimit_liminf}
\liminf_{m\to \infty}\,\bb P({D^{(1)} \btt \cdots \btt D^{(\ell)} \btt D^{(1)}})
    \;\geq\; \bb P\left(X_j> 0, j=1,\hdots, \ell \right).
\end{equation}
\end{enumerate}

When we combine \eqref{eq:condttslowgrowthpk_limsup} and \eqref{eq:condttslowgrowthpk_liminf}, we obtain that condition \eqref{eq:condttslowgrowthpk_lim} holds true, and then a combination of \eqref{eq:intransitivelimit_limsup} and \eqref{eq:intransitivelimit_liminf} yield \eqref{eq:intransitivelimit_ineq}.

Recall that the mean and variance of the $N_k$'s were computed in Lemma~\ref{lema:mean_var_Nk}, the quantities $f_k=f_k(m)$ and $\mb p_k=\mb p_k(m)$ are as in \eqref{deff:fkm} and \eqref{eq:defpk}, and for $k=1,\hdots, \ell$ denote
$$
v_k\;=\;v_k(m)\;\coloneqq\; \frac{1}{m^{3/2}}\var{N_k}^{1/2}\;=\;\sigma_k(\infty)(1+o(1))\,, \quad m\to\infty\,,
$$
so that $\widetilde N_k$ from \eqref{eq:def_tilde_Nk} reads as
$$
\widetilde N_k\;=\;\frac{N_k-m^2f_kf_{k+1}\mb p_k}{m^{3/2}v_k}\,,\quad k=1,\hdots,\ell.
$$
In an analogous way, and with Lemma~\ref{lem:meanvarEk} in mind, introduce the normalized version $\widetilde E_k$ of $E_k$ from \eqref{eq:def_Ek}, namely
$$
\widetilde E_k\;\coloneqq\; \frac{E_k-\bb E(E_k)}{m^{3/2}v_k^=}\,,\quad k=1,\hdots, \ell,
$$
with
\begin{equation}\label{eq:varEkdecayfinal}
v_k^=\;=\;v_k^=(m)\;\coloneqq\; \frac{1}{m^{3/2}}\var{E_k}^{1/2}\;=\;o(1)\,,
\end{equation}
where the last identity is valid thanks to Lemma~\ref{lem:notiemeanvariance}.
Finally, introduce the events
\begin{align*}
A_k & \;\coloneqq\; \left\{ \widetilde N_k>\frac{f_kf_{k+1}m^2}{m^{3/2}v_k}\left(\frac{1}{2}-\mb p_k\right)-\frac{1}{2m^{3/2}v_k}E_k \right\} \\
& \;=\; \left\{ \widetilde N_k>\frac{m^{1/2}f_kf_{k+1}}{v_k}\left(\frac{1}{2}-\mb p_k-\frac{1}{2}\mb p_k^=\right)-\frac{v_k^=}{2v_k}\widetilde E_k \right\}.
\end{align*}

These notations were introduced so that the identity \eqref{eq:ineqNkEk} writes simply as
$$
\bb P(D^{(1)}\btt \cdots \btt D^{(\ell)}\btt D^{(1)})\;=\;\bb P\left(A\right), \quad \text{where}\quad A\;\coloneqq\; \bigcap_{k=1}^\ell A_k\,.
$$
If we were to set $\widetilde E_k=0$, then the probability $\bb P(A)$ would be already suited for a direct application of Theorem \ref{teo:CLT}. However, in the general case that we are considering here, we need to estimate the possible contributions from the $E_k$'s in a more careful manner.

To that end, let us fix $\varepsilon>0$ and consider the events
\begin{align*}
& B_k(\varepsilon)\coloneqq \left\{\frac{v_k^=}{2v_k}|\widetilde E_k|>\varepsilon\right\}, \quad k=1,\hdots, \ell,\quad B(\varepsilon)\coloneqq \bigcup_{k=1}^\ell B_k(\varepsilon),\\
& C_k(\varepsilon)\coloneqq B_k(\varepsilon)^c = \left\{\frac{v_k^=}{2v_k}|\widetilde E_k|\leq \varepsilon\right\}, \quad k=1,\hdots, \ell,\quad C(\varepsilon)\coloneqq \bigcap_{k=1}^\ell C_k(\varepsilon)=B(\varepsilon)^c,
\end{align*}
and write
\begin{equation}\label{eq:PABC}
\bb P(A)\;=\;\bb P(A\cap B(\varepsilon))+\bb P (A\cap C(\varepsilon))\,.
\end{equation}
Given any $\varepsilon>0$, a simple union bound combined with Chebyshev's inequality gives
\begin{equation}\label{eq:Bepsilonzero}
\bb P(A\cap B(\varepsilon))\;\leq\; \bb P(B(\varepsilon)) \;\leq\; \frac{1}{4\varepsilon^2}\sum_{k=1}^\ell \left(\frac{v_k^=}{v_k}\right)^2.
\end{equation}
Thanks to \eqref{eq:varEkdecayfinal}, we thus conclude that
\begin{equation}\label{eq:estPABepsilon}
\bb P(A\cap B(\varepsilon))\stackrel{m\to\infty}{\longrightarrow} 0\,,\quad \text{for any } \varepsilon>0 \text{ fixed}.
\end{equation}
To handle the second term in the right-hand side of \eqref{eq:PABC}, for $t\in \bb R$ we introduce yet another event $D_k(t)$, namely
$$
D_k(t)\;\coloneqq\;
\left\{
\widetilde N_k>\frac{f_kf_{k+1}m^{1/2}}{v_k}\left(\frac{1}{2}-\mb p_k-\frac{1}{2}\mb p_k^=\right)-t
\right\}
,\; k=1,\hdots,\ell,\quad D(t)\;\coloneqq\; \bigcap_{k=1}^\ell D_k(t)\,.
$$
From the definition of $A_k,D_k(\varepsilon)$ and $C_k(\varepsilon)$, we obtain that
\begin{equation}\label{eq:inclusionACD}
A_k\cap C_k(\varepsilon)\;\subset\; D_k(\varepsilon)\cap C_k(\varepsilon), \quad k=1,\hdots, \ell.
\end{equation}
We now estimate the probability of the events on the right-hand side above. Conditioning, we compute
$$
\bb P(D(\varepsilon)\cap C(\varepsilon))\;=\;\bb  P(D(\varepsilon)\mid C(\varepsilon))\bb P(C(\varepsilon)) \;=\; \bb P(D(\varepsilon))-\bb P(D(\varepsilon)\mid C(\varepsilon)^c)\bb P(C(\varepsilon)^c)\,,
$$
and using that $C(\varepsilon)^c=B(\varepsilon)$ and \eqref{eq:Bepsilonzero}, we obtain
$$
\bb P(D(\varepsilon)\cap C(\varepsilon))\;=\;\bb P(D(\varepsilon))+o(1)\,,\quad \text{as }m\to \infty, \text{ for any } \varepsilon>0 \text{ fixed}.
$$
Finally, a combination of \eqref{eq:PABC}, \eqref{eq:estPABepsilon}, the inclusion \eqref{eq:inclusionACD} and this last estimate, we obtain that for any $\varepsilon>0$ fixed,
$$
\bb P(A)\;\leq\; \bb P(D(\varepsilon))+o(1),\quad \text{as }m\to\infty.
$$
Thus,
$$
\limsup_{m\to\infty} \bb P(A)\;\leq\; \limsup_{m\to\infty} \bb P(D(\varepsilon)),\quad \text{for any }\varepsilon>0.
$$
But from condition~\eqref{eq:condttslowgrowthpk_limsup} and Theorem~\ref{teo:CLT}, for any $\varepsilon>0$, the inequality
\begin{align*}
\limsup_{m\to\infty} \bb P(D(\varepsilon))
    &\;\leq\; \limsup_{m\to\infty} \bb P\Bigl(\widetilde N_k \geq -\frac{f_kf_{k+1}}{v_kr(m)}-\varepsilon ,\; k=1,\hdots, \ell \Bigr)  \\ 
    &\;\leq \; \bb P(X_k\geq -\varepsilon, \; k=1,\hdots, \ell)
\end{align*}
holds true, and the $\limsup$ estimate follows. The $\liminf$ estimate is
derived by a similar reasoning, since for every $\varepsilon>0$ it holds
\begin{equation*}
\bb P(A)
    \ge \bb P(A \cap C(\varepsilon))
    \ge \bb P(D(-\varepsilon) \cap C(\varepsilon))
    = \bb P(D(-\varepsilon)) + o(1),\quad \text{as $m \to \infty$}.
\end{equation*}
Hence, for every $\varepsilon>0$ we have
\begin{align*}
\liminf_{m\to\infty} \bb P(A)
    &\ge \liminf_{m\to\infty} \bb P(D(-\varepsilon))
    \ge \liminf_{m\to\infty}
    \bb P\Bigl(\widetilde N_k > \frac{f_kf_{k+1}}{v_kr(m)}+\varepsilon ,\;
    k=1,\hdots, \ell \Bigr)\\
    &\ge \bb P\Bigl(X_k > \varepsilon ,\;
    k=1,\hdots, \ell \Bigr),
\end{align*}
using condition~\eqref{eq:condttslowgrowthpk_liminf}. Taking $\varepsilon \downarrow 0$, the result follows.
\end{proof}

The proof of Theorem~\ref{thm:nointransitive} is now a simple consequence of a combination of our results.

\begin{proof}[Proof of Theorem~\ref{thm:nointransitive}]
Under the conditions of Theorem~\ref{thm:nointransitive}, we apply
    Theorem~\ref{thm:probabilitieszero} to conclude that the right-hand side of
    \eqref{eq:intransitivelimit_limsup} vanishes, and the proof is complete.
\end{proof}

\section{Proof of Theorem~\ref{teo:CLT}}\label{sec:proof_CLT}

We now move to the proof of the last standing Theorem~\ref{teo:CLT}. So during this section, $\{\mb D_m\}_m$ is a collection of $\ell$ independent random dice, each with number of faces $n_k=f_km$ satisfying Assumption~\ref{assumption:main}. Recall also that the random variables $\widetilde N_1(m),\hdots, \widetilde N_\ell(m)$ were introduced in \eqref{eq:def_Nk} and \eqref{eq:def_tilde_Nk}, they depend on the index $m$ of the sequence but we keep omitting this dependence and write $\widetilde N_k=\widetilde N_k(m)$. Likewise, the associated quantities $\mb p_k=\mb p_k(m), \mb q_k=\mb q_k(m), \mb r_k=\mb r_k(m), \mb s_k=\mb s_k(m),\sigma_k=\sigma_k(m), \gamma_k=\gamma_k(m)$, $k=1,\hdots, \ell$, were all defined by \eqref{eq:defpk}--\eqref{eq:def_gamma_k}; we also omit their dependence on $m$, and recall that they are instrumental in computing the leading terms in $\bb E(N_k), \var{N_k}$ and $\corr{N_{k-1},N_k}$ as in \eqref{eq:ENkvarNk}.

Thanks to Assumption~\ref{assumption:main} and Lemma~\ref{lema:mean_var_Nk}, we see that
\begin{equation}\label{eq:standarlization}
\widetilde N_k=\frac{N_k-n_kn_{k+1}\mb p_k}{m^{3/2}v_k}=\frac{N_k-m^2f_kf_{k+1}\mb p_k}{m^{3/2}v_k},\quad k=1,\hdots, \ell,
\end{equation}
with 
$$
\mb p_k=\mb p_k(\infty)+o(1),\quad v_k\coloneqq \frac{1}{m^{3/2}}\var{N_k}^{1/2}=\sigma_k(\infty)+o(1),\quad m\to \infty,\quad \sigma_k(\infty)>0.
$$
%

Our proof of the Central Limit Theorem will be based on the moment method, so for completeness we record here the moments of a general Gaussian random vector. For its statement, recall that
$$
n!! = \prod_{k=0}^{\lceil \frac{n}{2}-1\rceil}(n-2k)
$$
is the double factorial of a positive integer $n$, which is given by the product of all the positive integers up to $n$ that have the same parity as $n$.

\begin{proposition}
\label{prop:momentsmultinormal}
Let $X=(X_1,\cdots,X_\ell)^T\sim \mathcal N_{\ell}(0,\bm\Sigma)$ be a centered
Gaussian vector of size $\ell$ and covariance matrix $\bm \Sigma$ with rank
$r\geq 1$. Fix a column vector $\alpha=(\alpha_1,\cdots,\alpha_\ell)^T\in
\mathbb R^\ell$ for which $\alpha^T\bm \Sigma\alpha\neq 0$. Then
\begin{equation}\label{eq:momentsmultinormal}
\mathbb E\Bigl[ \Bigl(\sum_{j=1}^{\ell}\alpha_j X_j\Bigr)^{s}\Bigr]\;=\;
\begin{cases}
0, & \text{ if }s \text{ is odd}, \\
(\alpha^T \bm \Sigma \alpha)^{s/2} (s-1)!!, & \text{ if } s \text{ is even}.
\end{cases}
\end{equation}
\end{proposition}

\begin{proof}
The proof follows standard textbook arguments, we include it here for sake of completeness.
The matrix $\bm \Sigma$ is positive semi-definite, so it admits a Cholesky decomposition of the form
$$
\bm \Sigma\;=\;\bm L\bm L^T\,,
$$
where $\bm L$ is a real matrix of size $\ell\times r$ and $r$ is the rank of $\bm \Sigma$. At the level of the random variable $X$, it induces the identity
$$
X\;=\;\bm LZ\,,
$$
where $Z\sim \mathcal N_r(0,\bm I_r)$ is a normalized Gaussian vector of size $r$.
Now, set
$$
G\;=\;\frac{1}{(\alpha^T\bm\Sigma \alpha)^{1/2}} \alpha^T\bm L Z\,.
$$
Observe that $\alpha^T\Sigma \alpha>0$ so $G$ as above is well defined. In fact, $G$ is a linear combination of independent centered scalar Gaussian random variables, so $G$ is a centered Gaussian itself. Its variance is
$$
\mathbb E(G^2)\;=\;\mathbb E(GG^T)\;=\;\frac{1}{\alpha^T\bm \Sigma \alpha}\alpha^T\bm L\mathbb E[ZZ^T]\bm L^T\alpha\;=\;1\,.
$$
Hence $G$ is actually a standard Gaussian, so
$$
\mathbb E(G^s)\;=\;
\begin{cases}
0, & \text{ if }s \text{ is odd}, \\
(s-1)!!, & \text{ if } s \text{ is even}.
\end{cases}
$$
The proof is now completed by observing that the term inside the expectation on
    the left-hand side of \eqref{eq:momentsmultinormal} is $\left(\alpha
    X\alpha^T\right)^s =(\alpha^T \bm \Sigma \alpha)^{s/2} G^s.$
\end{proof}


From the Cramér--Wold Criterion, in order to prove Theorem~\ref{teo:CLT} it suffices to show that for any
$\alpha = (\alpha_1, \ldots, \alpha_\ell) \in \bb R^\ell$ we have
\begin{equation*}
\sum_{k=1}^{\ell} \alpha_k \widetilde{N}_k
    \;\overset{d}{\longrightarrow}\; \sum_{k=1}^\ell \alpha_k X_k \quad \text{as } m\to \infty,
\end{equation*}
where $X = (X_1, \ldots, X_\ell)^T \sim \mc N(0, \bm \Sigma)$ with $\bm
\Sigma$ as in Theorem~\ref{teo:CLT}.

To prove this, the method of moments will be used (see \cite[Theorem 3.12, page
109]{Durrett}), as the normal is a random variable uniquely determined by its
moments.  Thus, by Proposition~\ref{prop:momentsmultinormal}, we need to show
that for each $s\in \bb N$,
\begin{equation}
\label{eq:s_moment_tildeN}
\mathbb E\Bigl[ \Bigl(\sum_{k=1}^{\ell}\alpha_k \widetilde{N}_k\Bigr)^{s}\Bigr]\;\longrightarrow\;
\begin{cases}
0, & \text{ if }s \text{ is odd}, \\
(\alpha^T \bm \Sigma \alpha)^{s/2} (s-1)!!, & \text{ if } s \text{ is even},
\end{cases}
\end{equation}
as $m \to \infty$. 

The overall strategy we take is the following. The sum inside the expectation can be seen as a weighted sum over all pairs of dice faces that are being compared. We identify each term in this weighted sum with a sum over graphs with appropriate properties. This is done in Section~\ref{sec:sumtographs} below. 

Depending on certain properties of these graphs, they can either give an asymptotic negligible contribution or contribute to the leading order. In fact, we will show that at the end only graphs with a very particular structure contribute to the leading order of the sum. The second step of the proof consists in pinpointing the negligible contributions, and also identifying the structure of the graphs that give the leading contribution. This part is done in Section~\ref{sec:graphscounting}

The last part of the proof then consists in counting exactly the graphs that give the leading contributions, and this will be done in Section~\ref{sec:leadingcontrCLT}, which completes the proof of Theorem~\ref{teo:CLT}.

\subsection{From moments to combinatorics of graphs}
\label{sec:sumtographs} \hfill 

We now show how to identify the terms in the sum on the left-hand side of \eqref{eq:s_moment_tildeN} with a graph representation.

Using Lemma~\ref{lema:mean_var_Nk} and the definition of $N_k$ in \eqref{eq:def_Nk}, we write
\begin{equation*}
\alpha_k \widetilde{N}_k
    \;=\; \frac{\alpha_k}{\sigma_k m^{3/2}} (1+O(m^{-1/2})) \cdot
        \sum_{i=1}^{n_k} \sum_{j=1}^{n_{k+1}}
        (\ind_{D_i^{(k)} > D_j^{(k+1)}} - \mb p_k),
\end{equation*}
and therefore
\begin{equation}
\label{eq:sum_before_pow_s}
\sum_{k=1}^{\ell}\alpha_k \widetilde{N}_k
    \;=\; m^{-3/2} (1+O(m^{-1/2})) \sum_{k=1}^{\ell} \sum_{i=1}^{n_k} \sum_{j=1}^{n_{k+1}}
        \frac{\alpha_k}{\sigma_k} (\ind_{D_i^{(k)} > D_j^{(k+1)}} - \mb
        p_k)\,.
\end{equation}
Raising equation~\eqref{eq:sum_before_pow_s} to the power $s$ can be seen
combinatorially as choosing $s$ indexes $(k, i, j)$
from the triple sum above, multiplying their terms together and finally
summing over all possible choices. We now introduce a graph representation of this procedure. 

Define
\begin{equation}\label{eq:def_VEgraphG}
\begin{aligned}
V &\;\coloneqq\; \{(k,i); k \in [\ell], i \in [n_k]\},\\
E &\;\coloneqq\; \{e = \bigl((k,i),\, (k+1,j)\bigr):
        k\in [\ell], i \in [n_k], j\in [n_{k+1}]\}\,.
\end{aligned}
\end{equation}
The graph $\mc G = (V,E)$ has vertices representing all faces of all
dice and edges $e$ that represent the triples $(k, i, j)$ that appear in
equation~\eqref{eq:sum_before_pow_s}. Graph $\mc G$ already has some structure
inherited from the situation it encodes: it is clearly $\ell$-partite, with
parts $V_k := \{(k,i): i \in [n_k]\}$ and edges exist only between $V_k$ and
$V_{k+1}$.

Any choice $H = \{(k_t, i_t, j_t): t \in [s]\}$ of $s$ indices can be seen as
an ordered collection of $s$ (possibly repeated) edges of $\mc G$, and we refer
to the set of all possible $H$ as $\mc G_s$. Any fixed $H \in \mc G_s$
can be interpreted as a weighted subgraph of $\mc G$: for each edge $e \in \mc
G$, we assign the weight $w(e) = \# \{t \in [s]: (k_t, i_t, j_t)=e\}$, i.e.,
its multiplicity. 
For a graph $H \in \mc G_s$ introduce $\varphi(H)$ by
\begin{equation}
\label{eq:def_varphi_H}
\varphi(H)
    \;=\; \prod_{t \in [s]} \frac{\alpha_{k_t}}{\sigma_{k_t}}
        \bigl(\ind_{D^{(k_t)}_{i_t} > D^{(k_t+1)}_{j_t}} - \mb p_{k_t}\bigr).
\end{equation}

When we raise \eqref{eq:sum_before_pow_s} to the power $s$, we re-index the resulting sum on the right-hand side by $H\in \mc G_s$, and the factor $\varphi(H)$ is precisely the term in this sum that corresponds to a given graph $H\in\mc G_s$. Taking expectation, we thus obtain
\begin{equation}
\label{eq:s_moment_as_graph_sum}
\bb E\Bigl[ \Bigl(\sum_{k=1}^{\ell}\alpha_k \tilde{N}_k\Bigr)^{s}\Bigr]
    \;=\; m^{-\frac{3s}{2}} \bigl(1 + O(m^{-\frac{1}{2}})\bigr)
        \sum_{H \in \mc G_s} \bb E[\varphi(H)]\,.
\end{equation}
Equation~\eqref{eq:s_moment_as_graph_sum} expresses the expectation we want to compute in terms of a weighted sum over graphs, and the next step is to identify which structure on these graphs leads to leading and negligible asymptotic contributions as $m\to\infty$.

\subsection{Estimating the contributions from each class of graphs}
\label{sec:graphscounting}\hfill 

The next step is to estimate the terms inside the sum in \eqref{eq:s_moment_as_graph_sum}. The following claims emphasize
some of the main properties that will play a role in our computations.

\begin{claim}
\label{claim:boundedness}
Quantity $\varphi(H)$ is uniformly bounded for all $H \in \mc G_s$.
\end{claim}

\begin{proof}
Since $\sigma_k$ is bounded away from zero as $m$ tends to
infinity (see Assumption~\ref{assumption:main}-(ii)) we have
\begin{equation*}
\label{eq:varphi_H_unif_bound}
|\varphi(H)| \;\le\; \Bigl(2\max_{k \in [\ell]} \frac{\alpha_k}{\sigma_k}\Bigr)^s\,.
    \qedhere
\end{equation*}
\end{proof}

Let $H \in \mc G_s$. We say that edges $e_0$ and $\tilde{e}$ in $H$ are in the
same connected component if there is a sequence of edges $(e_j \in H; j \in [t])$
such that $e_{j-1}$ and $e_{j}$ have a vertex in common for every $j \in [t]$
and $e_t = \tilde{e}$. This forms an equivalence relation and we
partition the edges of $H$ into connected components. This is helpful to
take advantage of independence when evaluating expected values.
\begin{claim}
\label{claim:indep_components}
Suppose $H \in \mc G_s$ has $t$ connected components $H_1, \ldots, H_t$.
    Then 
    $$
    \bb E [\varphi(H)] = \prod_{i \in [t]} \bb E[\varphi(H_i)].
    $$
\end{claim}

\begin{proof}
It is immediate from the definitions, since for $i \neq j$ the random variables
$\varphi(H_i)$ and $\varphi(H_j)$ depend on disjoint sets of dice faces.
\end{proof}

Claims~\ref{claim:boundedness} and~\ref{claim:indep_components} allow us to
disregard the contribution of some classes graphs. In the next
two claims, we take advantage of the factor $m^{-\frac{3s}{2}}$ to conclude that the contribution of
graphs with too few or too many connected components is negligible.

\begin{claim}
\label{claim:few_components}
There are at most $K_1m^{(3s-1)/2}$ graphs in $\mc G_s$ with less than $s/2$
connected components, where $K_1$ does not depend on $m$.
\end{claim}

\begin{proof}
We give an upper bound on the number of graphs in $\mc G_s$ with $t$ connected
components. Define $f_{\max} \coloneqq \max_{k \in [\ell]} f_k$. The total number of
edges in $\mc G$ is
\begin{equation}
|E| \;=\; \sum_{k \in [\ell]} (f_k m) (f_{k+1} m)
    \;=\; m^2 \sum_{k \in [\ell]} f_k f_{k+1}
    \;\le\; (\ell f_{\max}^2) m^2\,.
\end{equation}
To count the number of graphs in $\mc G_s$, we begin by building such graphs $H\in \mc G_s$ in a specific ordering.
Let $H_j$ with $j \in [t]$ denote the $t$ connected components of a given $H$.
First, we choose one edge $e_j$ from $E$ for each $H_j$, without any
restriction. For these initial choices, we have at most
$((\ell f_{\max}^2) m^2)^t$ possibilities. Since $H$ has $s$ edges, we still
have to choose $s-t$ edges. For the remaining choices $e_j$ with
$j \in [s]\setminus [t]$ we will always choose $e_j$ so that it has some
vertex in common with some previously chosen $e_i$ with $i \in [j-1]$, to ensure that we
do not create any new connected components. Hence, on the second round of
choices, for choosing $e_j$ we have at most $(2(j-1))$ options for the
common vertex and at most $2f_{\max} m$ options for the other vertex. Hence,
we have at most
\begin{equation*}
((\ell f_{\max}^2) m^2)^t\,  (2(s-1)  2f_{\max} m)^{s-t}
    \;=\; K m^{t + s}
\end{equation*}
possibilities, where $K = K(\ell, s, t, f_{\max})$ is a positive constant.
Finally, observe that any graph $H\in \mc G_s$ with exactly $t$
connected components can have its edges reordered to a graph $H'\in \mc G_s$
so that the edges of $H'$ were chosen according to the procedure above.
It follows that the number of graphs in $\mc G_s$ with $t$ connected components is at most $s! K m^{t+s}$.
Therefore, there are at most
\begin{equation*}
s! K (m^{s+1}+m^{s+2}+\dots+m^{s+t})
    \;\le\; s! K t m^{s + t}
    \;\le\; {K}_1 m^{s + \frac{s-1}{2}}
\end{equation*}
graphs in $\mc G_s$ with less than $s/2$ connected components (as $t<s/2$, then
$t\le (s-1)/2$, because $t$ and $s$ are integer numbers). The positive constant
${K}_1$ does not depend on $m$, and the claim is proved.
\end{proof}

\begin{claim}
\label{claim:many_components}
If $H \in \mc G_s$ has more than $s/2$ connected components,
then $\esp{\varphi(H)}=0$.
\end{claim}

\begin{proof}
As there are more than $s/2$ connected components and only $s$ edges, at least
one of the components must be an isolated edge, say $H_1$ is just the edge
$((k,i),  (k+1,j))$. Then, we have
\begin{equation*}
\bb E [\varphi(H_1)]
    \;=\; \bb E\Bigl[
        \frac{\alpha_k}{\sigma_k}(\ind_{D^{(k)}_{i} > D^{(k+1)}_{j}} - \mb p_k)
        \Bigr]
    \;=\; \frac{\alpha_k}{\sigma_k} \bb E\bigl[
        \ind_{D^{(k)}_{i} > D^{(k+1)}_{j}} - \mb p_k
        \bigr]
    \;=\; 0,
\end{equation*}
where the expectation vanishes because of the definition of $\mb p_k$ in \eqref{eq:defpk}, and the result follows by Claim~\ref{claim:indep_components}.
\end{proof}

As a consequence of the claims above, we are able to pinpoint the leading order of the
$s$ moment in~\eqref{eq:s_moment_as_graph_sum} by focusing on a very specific
class of graphs in $\mc G_s$. We say that a connected component $H_j$ of a graph $H\in \mc G_s$ is a
\textbf{cherry} if it is composed by two distinct edges, and we say that a
graph $H\in \mc G_s$ is a \textbf{cherry graph} if all its connected components are
cherries. In particular, if $H\in \mc G_s$ is a cherry graph then $s$ must be even and $H$ must have exactly $s/2$ components.

The vertex of degree 2 in a cherry will be called joint and the other two will be called tips. Let us denote by $\mc C_s$ the set of all graphs $H \in \mc G_{2s}$ that are cherry graphs. In words, if $H \in \mc C_{s}$ then it has $s$ connected components of size two
with no repeating edges. 

It turns out that the leading contribution to the right-hand side of \eqref{eq:s_moment_as_graph_sum} comes precisely from cherry graphs, as claimed by our next result.

\begin{proposition}
\label{prop:s_moment_as_graph_sum_2}
For any positive integer $s$, the estimates
\begin{align}
\label{eq:s_moment_as_graph_sum_2_odd}
\bb E\Bigl[ \Bigl(\sum_{k=1}^{\ell}\alpha_k \widetilde{N}_k\Bigr)^{2s+1}\Bigr]
    &\;=\; O(m^{-1/2})\,,\\
\label{eq:s_moment_as_graph_sum_2_even}
\bb E\Bigl[ \Bigl(\sum_{k=1}^{\ell}\alpha_k \widetilde{N}_k\Bigr)^{2s}\Bigr]
    &\;=\; m^{-3s} \smash{\sum\limits_{H \in \mc C_{s}}} \bb E[\varphi(H)]
        + O(m^{-1})\,.
\end{align}
hold true as $m\to\infty$.
\end{proposition}
\begin{proof}
Let us estimate the $s$-moment via equation~\eqref{eq:s_moment_as_graph_sum}.
From Claims~\ref{claim:boundedness} and~\ref{claim:few_components} we conclude
that when estimating the sum in equation~\eqref{eq:s_moment_as_graph_sum} the
contribution of graphs in $\mc G_s$ with less than $s/2$ connected components is
too small when compared to $m^{3s/2}$. By Claim~\ref{claim:many_components},
the contribution of graphs with more than $s/2$ is precisely zero.
Hence, equation~\eqref{eq:s_moment_as_graph_sum_2_odd} is immediate: for
the $(2s+1)$-moment we can write
\begin{align*}
m^{-\frac{3}{2}(2s+1)}\smash{\sum\limits_{H \in \mc G_{2s+1}}} \bb E[\varphi(H)]
    &\;=\; m^{-\frac{3}{2}(2s+1)} \smash{\sum\limits_{1 \le t \le s}}\,
    \sum\limits_{\substack{\text{\scriptsize $H$ with $t$ connected}\\ \text{\scriptsize components}}}
    \hspace{-3mm} \bb E[\varphi(H)]\\
    &\;\le\; m^{-3s - 3/2} \cdot K m^{3s+1}
    =  K m^{-1/2}\,.
\end{align*}

When estimating the $2s$-moment, the same argument shows that we only have to
worry with the contribution of graphs with exactly $s$ connected
components. By Claim~\ref{claim:indep_components} if $H$ has some component
with only one edge then $\bb E[\varphi(H)] = 0$. Consequently, for a non-zero
contribution, each of the $s$ components must have at least 2 edges. But then
we already have $2s$ edges in total, and we conclude that each component has
exactly 2 edges. In principle, we can have multiple edges in such
graphs. However, another simple counting argument shows that
the number of such graphs containing at least one multiple edge is
at most $K m^{3s-1}$. This implies we can focus on the sum over $H \in \mc
C_s$, as claimed in \eqref{eq:s_moment_as_graph_sum_2_even}.
\end{proof}

With Proposition~\eqref{prop:s_moment_as_graph_sum_2} at hand, the remaining step is to estimate the sum over cherry graphs in the right-hand side of \eqref{eq:s_moment_as_graph_sum_2_even}, we will see that this remaining sum is in fact $\Theta(m^{3s})$,
so the leading order is indeed given by it, and we will in fact be able to compute its contribution precisely.

\subsection{Computing the leading contribution, and the conclusion of the proof of Theorem~\ref{teo:CLT}}
\label{sec:leadingcontrCLT} \hfill 

What remains is to count all the cherry graphs $H \in \mc C_s$ and compute
$\bb E[\varphi(H)]$. Since cherries are disjoint, we break any $H\in \mc C_s$ into $s$
cherries, which we denote $H_1, \ldots, H_s$. For a cherry $H_j$ the value
of $\bb E[\varphi(H_j)]$ will depend on which dice are used for its joints and tips.
We say that cherry $H_j$ has:
\begin{description}
\item[Type $(k, 1)$] If its joint is on die $D^{(k)}$, one tip is on
    $D^{(k-1)}$ and the other is on $D^{(k+1)}$.
\item[Type $(k, 2)$] If its joint is on die $D^{(k)}$, and both tips are on $D^{(k+1)}$.
\item[Type $(k, 3)$] If its joint is on die $D^{(k)}$, and both tips are on $D^{(k-1)}$.
\end{description}

By the construction of the graph $\mc G$ in \eqref{eq:def_VEgraphG}, these are the only cherries that can occur as components of a graph $H\in \mc C_s$. It is straightforward to compute $\bb E[\varphi(H_j)]$.
\begin{proposition}
\label{prop:varphi_cherries}
If cherry $H_j$ has type $(k,t)$ then $\bb E[\varphi(H_j)]$ depends only on
$(k,t)$. Denoting its value by $\varphi_{k,t}$, we have
\begin{equation}
\label{eq:def_varphi_constants}
\varphi_{k,t}
    \;\coloneqq\; \bb E[\varphi(H_j)] \;=\;
\begin{dcases}
\frac{\alpha_{k-1} \alpha_{k}}{\sigma_{k-1} \sigma_{k}} 
    (\mb s_k - \mb p_{k-1} \mb p_k)
    &\text{if }t=1;\\
\left(\frac{\alpha_{k}}{\sigma_{k}}\right)^2 
    (\mb r_k - \mb p_k^2)
    &\text{if }t=2;\\
\left(\frac{\alpha_{k-1}}{\sigma_{k-1}}\right)^2 
    (\mb q_k - \mb p_{k-1}^2)
    &\text{if }t=3.
\end{dcases}
\end{equation}
\end{proposition}

\begin{proof}
It is straightforward from the definition of $\varphi(H_j)$ given
in~\eqref{eq:def_varphi_H} and the definitions of $\mb p_k, \mb q_k, \mb r_k,
\mb s_k$ in \eqref{eq:defpk}--\eqref{eq:defsk}.
\end{proof}

Since $k \in [\ell]$, we have in total $3\ell$ different types of cherries. Recall that
the total number of cherries is $s$. It is useful to classify $H \in \mc C_s$
with respect to the number of occurrences of each type. We define $M_{k,t}=M_{k,t}(H)$
as the number of cherries of type $(k,t)$ on the cherry graph $H$, we encode these numbers in the matrix
$M = (M_{k,t})_{\ell \times 3}$, and say that $H$ has type $M$. Observe that $M_{k,t} \in \bb N$ and
$\sum_{k,t} M_{k,t} = s$. 

With this codification in mind, for a cherry graph $H$ of type $M$ we have
\begin{equation}
\label{eq:varphi_cherries_by_type}
\bb E[\varphi(H)]
    \;=\; \prod_{j=1}^{s} \bb E[\varphi(H_j)]
    \;=\; \prod_{k,t} \varphi_{k,t}^{M_{k,t}}\,.
\end{equation}

Finally, to estimate the sum over all $H \in \mc C_s$ we partition the cherry
graphs according to the possible types. Let $\mc C_s(M)$ denote the set of
all cherry graphs $H \in \mc C_s$ of type $M$.
We need estimates on the number of elements of $\mc C_s(M)$ for each $M$.
\begin{lemma}
\label{lema:cherry_graphs_by_type}
For each cherry type $(k,t)$, define
\begin{equation}
\label{eq:def_cherry_constants}
c_{k,t} \;=\;
    \begin{cases}
    f_{k-1} f_k f_{k+1}
        &\text{if $t=1$};\\
    \frac{1}{2} f_{k} f_{k+1}^{2}
        &\text{if $t=2$};\\
    \frac{1}{2} f_{k} f_{k-1}^{2}
        &\text{if $t=3$}.
    \end{cases}
\end{equation}
As $m$ tends to infinity, the size $|\mc C_s(M)|$ of $\mc C_s(M)$ satisfies
\begin{equation}
\label{eq:cherry_graphs_by_type}
|\mc C_s(M)|
    \; = \;   m^{3s}\Bigl[ \prod_{k,t} \frac{c_{k,t}^{M_{k,t}}}{M_{k,t}!}\Bigr]  \left((2s)!\right)(1+O(1/m))
\end{equation}
\end{lemma}

Since $f_k=f_k(m)$ depends on $m$, we also have that $c_{k,t}=c_{k,t}(m)$ depends on $m$, but in virtue of Assumption~\ref{assumption:main} each $c_{k,t}$ has a nonzero limit as $m\to\infty$.

\begin{proof}
We begin by counting the number of cherries of a given specific type, considering
that its edges are \emph{not ordered}. Recall that $n_k=f_k m$ denotes the number of faces in the die $D^{(k)}$, and let us define
\begin{equation}
\label{eq:def_number_cherries_by_type}
C_{k,t}(n_1,\ldots, n_\ell)
    \;\coloneqq\; \frac{1}{2} | \{H \in \mc G_{2}: \text{$H$ is a cherry of type
    $(k,t)$}\}|\,,
\end{equation}
where the factor $\frac{1}{2}$ is precisely to disregard the order of edges in
a cherry. By a simple counting argument, we have that
\begin{equation}
\label{eq:number_cherries_by_type}
C_{k,t}(n_1, \ldots, n_\ell)
    \;=\;
\begin{dcases}
n_{k-1} n_k n_{k+1}
    &\text{if $t=1$};\\
    n_{k} \binom{n_{k+1}}{2}
    &\text{if $t=2$};\\
n_{k} \binom{n_{k-1}}{2}
    &\text{if $t=3$}.
\end{dcases}
\end{equation}
Therefore, the estimate
$$
C_{k,t}(n_1, \ldots, n_\ell) = c_{k,t} m^3+O(m^2),\quad \text{as }m\to \infty,
$$
is valid, where $c_{k,t}$ are the values in \eqref{eq:def_cherry_constants}.

Now, let us fix a type $M = (M_{k,t})$. First, we compute in how many ways we
can choose an unordered collection of $s$ cherries with exactly $M_{k,t}$
occurrences of each type $(k,t)$. Given $M$, we will choose its
cherries one by one following the sequence of types
$\{(k_j, t_j): j \in [s]\}$ in lexicographic order.

The first cherry, with type $(k_1,t_1)$, is chosen from
all possible edges of $\mc G$. When choosing the following cherries, we have
to successively remove the vertices that appear in the previous cherries, to
ensure disjointness. Hence, when choosing the vertices of cherry $(k_j,t_j)$
we have $C_{k_j,t_j}(n^{(j)}_1, \ldots, n^{(j)}_\ell)$ options, where $n^{(j)}_i$
is the number of faces of die $D^{(i)}$ that do not appear in the $j-1$
previously chosen cherries. It is clear that $(n^{(j)}_i)$ will depend on
the sequence $\{(k_j, t_j)\}$. However, for our estimates it is enough to
notice that since we only choose $s$ cherries, we have $n^{(j)}_i = f_i m+O(1)$.
Finally, the above procedure chooses the $s$ cherries following the ordering
$\{(k_j, t_j)\}$. Hence, the number of choices of an unordered collection of
$s$ cherries is given by
\begin{equation*}
\frac{\prod_{j \in [s]} |C_{k_j, t_j}(n^{(j)}_1, \ldots, n^{(j)}_\ell)|}
    {\prod_{k,t} M_{k,t}!}
    \;=\;
\frac{\prod_{j \in [s]} \left(c_{k_j, t_j} m^3+O(m^2)\right)}
    {\prod_{k,t} M_{k,t}!}
    \;=\; \Bigl[ \prod_{k,t} \frac{c_{k,t}^{M_{k,t}}}{M_{k,t}!}\Bigr] m^{3s}(1+O(1/m))\,.
\end{equation*}
To conclude the argument, just notice that when summing over $H \in \mc C_s$
we are actually summing over all fixed unordered collection and considering all
possible permutations of the $2s$ edges that compose $H$. The estimate in
equation~\eqref{eq:cherry_graphs_by_type} follows.
\end{proof}

Now, we proceed with the estimate in
equation~\eqref{eq:s_moment_as_graph_sum_2_even}. Breaking the sum on the
right-hand side with respect to the type $M$ of the cherry graphs $H \in \mc
C_s$ and using equation~\eqref{eq:varphi_cherries_by_type}, we have that
\begin{align}
\bb E\Bigl[ \Bigl(\sum_{k=1}^{\ell}\alpha_k \widetilde{N}_k\Bigr)^{2s}\Bigr]
    &\;=\; m^{-3s} \smash{\sum\limits_{H \in \mc C_s}} \bb E[\varphi(H)]
        + O(m^{-1})\nonumber\\
    &\;=\; m^{-3s} \smash{\sum_M\sum\limits_{H \in \mc C_s(M)}}
        \prod_{k,t} \varphi_{k,t}^{M_{k,t}} + O(m^{-1})\nonumber\\
\label{eq:2s_moment_by_M}
    &\;=\; \smash{\sum_M} (2s)! \cdot
    \Bigl[ \prod_{k,t} \frac{(c_{k,t}\varphi_{k,t})^{M_{k,t}}}{M_{k,t}!}\Bigr]
    + O(m^{-1})\,.
\end{align}
To obtain a more meaningful expression, we recognize the sum over $M$ as a sum to the $s$ power. Indeed, recall that $M = (M_{k,t})$
is such that $M_{k,t} \in \bb N$ must sum to $s$. Hence, we can write
\begin{align}
\smash{\sum_M} (2s)! \cdot
    \Bigl[ \prod_{k,t} \frac{(c_{k,t}\varphi_{k,t})^{M_{k,t}}}{M_{k,t}!}\Bigr]
    &\;=\; \frac{(2s)!}{s!}\smash{\sum_{(M_{k,t}); \sum M_{k,t}=s}}\;
        \frac{s!}{\prod\limits_{k,t} M_{k,t}!} 
        (c_{k,t}\varphi_{k,t})^{M_{k,t}}\nonumber\\
    &\;=\; \frac{(2s)!}{s!} \Bigl(\sum_{k,t} c_{k,t} \varphi_{k,t}\Bigr)^s\nonumber\\
\label{eq:2s_moment_by_M_2}
    &\;=\; (2s-1)!! \Bigl(\sum_{k,t} 2c_{k,t} \varphi_{k,t}\Bigr)^s,
\end{align}
where we used that $(2s-1)!! = \frac{(2s)!}{s!2^s}$.

Using equation~\eqref{eq:2s_moment_by_M} the next step is to identify the limit of this $2s$-moment. From equations~\eqref{eq:def_cherry_constants}
and~\eqref{eq:def_varphi_constants} we have that 
\begin{equation}\label{eq:cktphikt}
2 c_{k,t} \varphi_{k,t}
    \;=\;
\begin{cases}
2 f_{k-1} f_k f_{k+1} \frac{\alpha_{k-1} \alpha_{k}}{\sigma_{k-1} \sigma_{k}}
    \cdot (\mb s_k - \mb p_{k-1} \mb p_k)
    &\text{if $t=1$;}\\
f_k f_{k+1}^2 \bigl(\frac{\alpha_{k}}{\sigma_{k}}\bigr)^2 \cdot
    (\mb r_k - \mb p_k^2)
    &\text{if $t=2$;}\\
f_k f_{k-1}^2 \bigl(\frac{\alpha_{k-1}}{\sigma_{k-1}}\bigr)^2 \cdot
    (\mb q_k - \mb p_{k-1}^2)
    &\text{if $t=3$,}
\end{cases}
\end{equation}
and we can recognize the sum over $(k,t) \in [\ell]\times [3]$ as a quadratic form
in the vector $\alpha = (\alpha_1, \ldots, \alpha_\ell)^T$. The coefficient of $\alpha_k^2$ is given by
\begin{equation*}
\frac{1}{\sigma_k^2} \Bigl[
    f_k   f_{k+1}^2 (\mb r_{k  } - \mb p_k^2) +
    f_k^2 f_{k+1}   (\mb q_{k+1} - \mb p_k^2)
    \Bigr] \;=\; 1\,,
\end{equation*}
recalling the definition of $\sigma_k$ in~\eqref{eq:def_sigma_k}.
The coefficient of $\alpha_{k-1} \alpha_{k}$ is precisely the value $\gamma_k=\gamma_k(m)$ given by~\eqref{eq:def_gamma_k}
\begin{equation*}
\gamma_{k}
    \;=\; \frac{1}{\sigma_{k-1} \sigma_k} f_{k-1} f_k f_{k+1}
    (\mb s_{k} - \mb p_{k-1} \mb p_k)\,.\qedhere
\end{equation*}

Writing $\alpha=(\alpha_1,\hdots, \alpha_\ell)^T$, and defining
$$
\Sigma(m) \, \coloneqq \, \left(
\begin{array}{cccccccc}
1 & \gamma_2(m) & 0 & \cdots  & 0 &\gamma_1(m)\\
\gamma_2(m) & 1 & \gamma_3(m) & \cdots & 0 & 0\\
0 & \gamma_3(m) & 1 & \cdots& 0 & 0\\
\vdots & \vdots & \vdots & \ddots & \vdots & \vdots\\
0 & 0 & 0 & \cdots & 1 & \gamma_{\ell}(m)\\
\gamma_1(m) & 0 & 0 & \cdots & \gamma_{\ell}(m) & 1
\end{array}\right),
$$
with $\gamma_k(m)$ as in \eqref{eq:def_gamma_k}, we just unraveled the identity
$$
\sum_{k,t} 2c_{k,t} \varphi_{k,t} = \alpha^T\Sigma(m)\alpha.
$$
By Assumption~\ref{assumption:main}, we learn that $\Sigma(m)$ converges to the matrix $\Sigma$ in \eqref{eq:covariancematrixCLT}. This last convergence thus shows \eqref{eq:s_moment_tildeN}, and concludes the proof of Theorem~\ref{teo:CLT}.

\section{Asymptotic behavior of the number of intransitive strings}\label{sec:stringsAsymptotics}

In this section we prove Theorem~\ref{thm:ellvalue}. We continue using the notation and notions introduced in Section~\ref{sec:deterministic} extensively. In particular, we will work with both sets $\mc W_{\btt, \ell}(n)$ and $\mc D_{\btt, \ell}(n)$.

To prove Theorem~\ref{thm:ellvalue} we will find lower and upper bounds for $|\mc W_{\btt, \ell}(n)|$. However, the structure of the set $\mc W_{\btt, \ell}(n)$ is too cumbersome for a direct combinatorial analysis. We circumvent this by translating the set of words into a set of dice that behave well, allowing us to use Theorem~\ref{teo:CLT} in its full force. 

Any collection of dice $\mb D=(D^{(1)},\hdots, D^{(\ell)})$, each with $n$ faces, and with any two faces of any two dice being distinct, may be embedded in the set $\mc D_\ell(n)$ from \eqref{eq:defDellbtt} in a natural way: the smallest among the numbers $D^{j}_i$ is replaced by the number $1$, the second smallest among $D^{j}_i$ is replaced by $2$, and so forth. 

Denote by $\mb D_{U, \ell}(n)=(D^{(1)}, \dots, D^{(\ell)})$ a collection of random dice, such that the faces $D^{(i)}_j$ are all i.i.d. uniform random variables in $(0, 1)$. If $(i, j) \neq (k, l)$, then $\bb P(D^{(i)}_j=D^{(k)}_l)=0$, so that any two faces within the collection $\mb D_{U,\ell}(n)$ are distinct almost surely. Thanks to the argument in the previous paragraph, we view $\mb D_{U,\ell}(n)$ as an element of $\mc D_\ell(n)$, which generates a random word $\mb W_{U,\ell}(n)\coloneqq\pi (\mb D_{U,\ell}(n))\in \mc W_\ell(n)$, where $\pi$ is the canonical bijection between $\mc D_\ell(n)$ and $\mc W_\ell(n)$ constructed at the beginning of Section~\ref{sec:dicewords}.

A routine symmetry argument shows that $\pi(\mb D_{U, \ell}(n))$ can be any word in $\mc W_{\ell}(n)$ with equal probability, so that $\pi$ induces the uniform distribution in $\mc W_{\ell}(n)$. We will use this observation extensively in what follows.

By Theorem~\ref{thm:nointransitive} we infer in particular that
    \begin{equation}\label{eq:nullProbability}
        \frac{|\mc W_{\btt, \ell}(n)|}{|\mc W_{\ell}(n)|}= \bb P\left({D^{(1)} \btt \cdots \btt D^{(\ell)} \btt D^{(1)}}\right) \to 0 \quad\text{as}\quad n\to\infty.
    \end{equation}
On the other hand, by Theorem~\ref{thm2}, 
    \[
        |\mc W_{\btt, \ell}(n)| = \ee^{nL(\ell)+o(n)}.
    \]
    As indicated by \eqref{eq:nullProbability}, the number of intransitive words is rather small compared to the total number of words. Furthermore, the number of intransitive words obtained by concatenating two other intransitive words cannot substantially exceed the product of the sizes of each set, so there is little hope of estimating $L(\ell)$ using this approach alone. However, we can improve this algorithm to increase the count of intransitive words significantly. By taking almost intransitive words from a large subset of $\mc W_\ell(n)$ and concatenating them with highly intransitive words, we will construct a large subset of intransitive words of size $\sim \ee^{n\ell \log \ell}$.

Introduce
    \[
        \mc Q_\ell(n)\coloneqq\left\{\mb W \in\mc W_\ell(n): N_{k}(\mb W)>\frac{n^2}{2}-\frac{n^{3/2}}{2}\left(1+\frac{1}{2n}\right)^{1/2},\ k=1, \dots, \ell\right\}.
    \]
We use Lemma~\ref{lem:prqscontdistr} and compute explicitly
    $$
    \wt N_{k}(\mb W_{U,\ell}(n))=\frac{N_{k}(\mb W_{U,\ell}(n))-n^2/2}{n^{3/2}(1+1/2n)^{1/2}/\sqrt{6}},\quad k=1,\hdots, \ell.
    $$
    Stressing that we equip $\mc W_{\btt, \ell}(n)$ with the uniform distribution, we see that
    \begin{equation}\label{eq:sizencQncW}
    \frac{|\mc Q_\ell(n)|}{|\mc W_\ell(n)|}=\bb P\left(\wt N_{k}(\mb W_{U,\ell}(n))>-\sqrt{6}/2,\ k=1, \dots, \ell \right).
    \end{equation}
    For a word $\mb W$ to be intransitive, we must have $N_k(\mb W)>n^2/2$, so the set $\mc Q_\ell(n)$ contains words that are not intransitive. Using our CLT, namely Theorem~\ref{teo:CLT}, we now show that $\mc Q_\ell(n)$ is rather large, and later we will use composition of words to construct a new set of intransitive words with the same size as $\mc Q_\ell(n)$.
    
\begin{lemma}
For any $\ell\geq 3$, the limit 
    $$
    \lim_{n\to\infty} \frac{\log|\mc Q_{\ell}(n)|}{n} = \ell\log\ell
    $$
is valid.
\end{lemma}
\begin{proof}
Applying Theorem~\ref{teo:CLT} to the right-hand side of \eqref{eq:sizencQncW}, we see that
\begin{equation}\label{eq:comparisonQW}
\frac{|\mc Q_\ell(n)|}{|\mc W_\ell(n)|}\to \bb P \left( X_j\geq -\sqrt{6}/2,\; j=1,\hdots, \ell \right),
\end{equation}
where $(X_1,\hdots, X_\ell)$ is a centered Gaussian vector. The covariance matrix \eqref{eq:covariancematrixCLT} must be computed having in mind that the underlying vectors $(\wt N_1(n),\hdots, \wt N_\ell(n))$ are all coming from i.i.d. dice $\mb D_{U,\ell}(n)$, and thanks to Lemma~\ref{lem:prqscontdistr} this matrix is explicitly given by $\gamma_k(\infty)=-1/2$, for every $k$.

Because the vector $(X_1,\hdots X_\ell)$ is centered, we must have that $\bb P((X_1,\hdots, X_\ell)\in  B_\varepsilon )>0$ for every ball $B_\varepsilon\subset \bb R^\ell$ centered at the origin with radius $\varepsilon>0$. By choosing $\varepsilon>0$ sufficiently small, we can make sure that $B_\varepsilon\subset \{(x_1,\hdots, x_\ell)\in \bb R^\ell \mid x_i\geq -\sqrt{6}/2, \; i=1,\hdots, \ell\}$, and therefore
$$
\bb P \left( X_j\geq -\sqrt{6}/2,\; j=1,\hdots, \ell \right)\geq \bb P ((X_1,\hdots X_\ell)\in B_\varepsilon)>0
$$
Having in mind that $(\log|\mc W_{\ell}(n)|)/n=(\log|\mc D_{\ell}(n)|)/n \to \ell \log \ell$ (see \eqref{eq:subexpondecay}), the result now follows from \eqref{eq:comparisonQW}.
\end{proof}

As said, we will concatenate words in $\mc Q_\ell(n)$ with a new word $\mb S$ to produce a large set of intransitive words. Since words in $\mc Q_\ell(n)$ are not necessarily intransitive, we need to choose this new word $\mb S$ to be ``highly intransitive'', to counterbalance the transitivity of words in $\mc Q_\ell(n)$. We present such a word $\mb S$ in the next lemma.
    \begin{lemma}\label{lem:specialseqwords}
        There exists an increasing sequence $(n_k)_k\subset\mathbb N$ with the following property: for each $k$ there is $\mb S_k\in \mc W_{\btt, \ell}(n_k)$  fulfilling
        \[
            N_{k}(\mb S_k) \ge \frac{n_k^2}{2}+\frac{1}{2}n_k^{\frac{3}{2}(1+\frac{1}{18})}, \quad \text{for} \quad k=1, \hdots, \ell,
        \]
        and for each $k$, $n_k^{1+{1}/{36}}\in\mathbb{N}$.
    \end{lemma}
    
    \begin{proof}
            For each $k\in \bb N$, set $n_k\coloneqq (8k^3)^{12}$. With $t\coloneqq 8k^3\in \bb N$, the word 
            \[
                \mb S_k \coloneqq D^{(\ell)}_{[t^7]}D^{(1)}_{[t^{12}/2]}D^{(2)}_{[t^{12}-t^7]}D^{(3)}_{[t^{12}]}\cdots D^{(\ell-1)}_{[t^{12}]}D^{(\ell)}_{[t^{12}-t^7]}D^{(1)}_{[t^{12}/2]}D^{(2)}_{[t^7]},
            \]
            where 
            \[
                D^{(i)}_{[j]}\coloneqq \underbrace{D^{(i)}\dots D^{(i)}}_{j},
            \]
            fulfils all the conditions of the statement.
    \end{proof}

    We are finally ready to prove Theorem~\ref{thm2}.

    \begin{proof}[Proof of Theorem~\ref{thm2}]

As observed in \eqref{eq:subexpondecay}, we already know that $L(\ell)\le \ell\log\ell$. The remaining task is to prove the converse inequality.
        
Let $(\mb S_k)$ be the sequence of words from Lemma~\ref{lem:specialseqwords}, and set $m_k\coloneqq n_k^{1+1/36}$. We claim that for any $k$ sufficiently large, given any word $\mb W \in \mc Q_\ell(m_k)$, the concatenated word $\mb W\mb S_k$ is intransitive, that is, $\mb W\mb S_k\in \mc W_{\btt, \ell}(n_k+m_k)$. In fact, for each $i=1,\dots, \ell$ and for sufficiently large $k$, and recalling \eqref{eq:deffNijD} and that $N_i(\cdot)=N_{i,i+1}(\cdot)$, we compute
        \begin{align*}
            2N_{i, i+1}(\mb W \mb S_k) &\;=\; 2N_{i}(\mb W)+2N_{i}(\mb S_k)+2m_kn_k\\
            &\; > \; m_k^2-m_k^{\frac{3}{2}}\left(1+\frac{1}{2m_k}\right)^{\frac{1}{2}}+n_k^2+n_k^{\frac{3}{2}(1+\frac{1}{18})}+2m_kn_k\\
            &\; = \; (m_k+n_k)^2+n_k^{\frac{3}{2}(1+\frac{1}{18})}-n_k^{\frac{3}{2}(1+\frac{1}{36})}\left(1+\frac{1}{2m_k}\right)^{1/2}\\
            &\; = \; (m_k+n_k)^2 + n_k^{\frac{3}{2}(1+\frac{1}{36})}\left(n_k^{1/24}-\left(1+\frac{1}{2m_k}\right)^{1/2} \right)\\
            &\; \ge \; (m_k+n_k)^2.
        \end{align*}
    
As the map $\mb W \mapsto \mb W\mb  S_k$ is one-to-one, we have $|\mc W_{\btt, \ell}(m_k+n_k)|\ge|\mc Q_{\ell}(m_k)|$, so
        \begin{align*}
            \ell\log\ell&\;=\;\lim_{k\to\infty}\frac{\log|\mc Q_{\ell}(m_k)|}{m_k}\\
            &\;\le\; \lim_{k\to\infty} \frac{\log|\mc W_{\btt, \ell}(m_k+n_k)|}{m_k}\\
            &\;=\; \lim_{k\to\infty} \frac{\log|\mc W_{\btt, \ell}(m_k+n_k)|}{m_k+n_k}\cdot \frac{m_k+n_k}{m_k}
            \;=\; L(\ell)\lim_{k\to\infty}\frac{n_k^{1+1/36}+n_k}{n_k^{1+1/36}}
            \;=\; L(\ell),
        \end{align*}
        which completes the proof.
    \end{proof}

    \bibliographystyle{plain}
    \bibliography{main}

\end{document}